\documentclass[11pt,reqno]{amsart}

\usepackage{graphicx}

\usepackage{xspace}
\usepackage{amsmath,amsthm,amsfonts,amssymb,latexsym,mathrsfs,color,extarrows}
\usepackage{hyperref}

\newtheorem{theorem}{Theorem}
\newtheorem{corollary}[theorem]{Corollary}
\newtheorem{proposition}[theorem]{Proposition}
\newtheorem{conjecture}[theorem]{Conjecture}

\newtheorem{lemma}[theorem]{Lemma}
\newtheorem{definition}[theorem]{Definition}
\newtheorem{example}[theorem]{Example}

\newcommand{\diag}{{\rm diag\,}}

\newcommand{\cval}{{\rm cval\,}}
\newcommand{\cpk}{{\rm cpk\,}}

\newcommand{\Expk}{{\rm Expk\,}}
\newcommand{\Ubdes}{{\rm Ubdes\,}}
\newcommand{\bddes}{{\rm bddes\,}}
\newcommand{\expk}{{\rm expk\,}}
\newcommand{\ubdes}{{\rm ubdes\,}}

\newcommand{\LS}{{\rm LS\,}}
\newcommand{\cdd}{{\rm cdd\,}}
\newcommand{\Ddes}{{\rm Ddes\,}}
\newcommand{\Laplat}{{\rm Laplat\,}}
\newcommand{\Orb}{{\rm Orb\,}}
\newcommand{\Des}{{\rm Des\,}}
\newcommand{\Dasc}{{\rm Dasc\,}}
\newcommand{\cda}{{\rm cda\,}}

\newcommand{\JSP}{{\rm JSP\,}}
\newcommand{\JS}{{\rm JS\,}}

\newcommand{\dasc}{{\rm dasc\,}}
\newcommand{\desp}{{\rm desp\,}}

\newcommand{\plat}{{\rm plat\,}}

\newcommand{\lap}{{\rm laplat\,}}

\newcommand{\single}{{\rm single\,}}

\newcommand{\ddes}{{\rm ddes\,}}
\newcommand{\des}{{\rm des\,}}

\newcommand{\exc}{{\rm exc\,}}
\newcommand{\aexc}{{\rm aexc\,}}
\newcommand{\we}{{\rm wexc\,}}

\newcommand{\fix}{{\rm fix\,}}
\newcommand{\dc}{{\rm dc\,}}

\newcommand{\mdn}{\mathcal{D}}

\newcommand{\msn}{\mathfrak{S}_n}

\newcommand{\ms}{\mathfrak{S}}

\newcommand{\lrf}[1]{\lfloor #1\rfloor}

\newcommand{\mq}{\mathcal{Q}}
\newcommand{\mqn}{\mathcal{Q}_n}

\newcommand{\asc}{{\rm asc\,}}

\newcommand{\Eulerian}[2]{\genfrac{<}{>}{0pt}{}{#1}{#2}}
\newcommand{\Stirling}[2]{\genfrac{\{}{\}}{0pt}{}{#1}{#2}}

\title{$\gamma$-positivity and partial $\gamma$-positivity of descent-type polynomials}
\author[S.-M.~Ma]{Shi-Mei Ma}
\address{School of Mathematics and Statistics,
        Northeastern University at Qinhuangdao,
         Hebei 066000, P.R. China}
\email{shimeimapapers@163.com (S.-M. Ma)}
\author[J.~Ma]{Jun Ma}
\address{Department of mathematics, Shanghai jiao tong university, Shanghai, P.R. China}
\email{majun904@sjtu.edu.cn(J.~Ma)}
\author[Y.-N. Yeh]{Yeong-Nan Yeh}
\address{Institute of Mathematics,
        Academia Sinica, Taipei, Taiwan}
\email{mayeh@math.sinica.edu.tw (Y.-N. Yeh)}
\begin{document}

\maketitle
\begin{abstract}
In this paper, we study $\gamma$-positivity of descent-type polynomials
by introducing the change of context-free grammars method.
We first present grammatical proofs of the $\gamma$-positivity of the
Eulerian polynomials, type $B$ Eulerian polynomials, derangement polynomials, Narayana polynomials and type $B$ Narayana polynomials. We then provide partial
$\gamma$-positive expansions for several multivariate
polynomials associated to Stirling permutations,
Legendre-Stirling permutations, Jacobi-Stirling permutations and type $B$ derangements, and the recurrences
for the partial $\gamma$-coefficients of these expansions are also obtained. Moreover, we define variants of the Foata-Strehl group action which are used to give combinatorial interpretations for the coefficients of most of these partial $\gamma$-positive expansions.
\bigskip

\end{abstract}
\date{\today}
\section{Introduction}
\hspace*{\parindent}
Let $f(x)=\sum_{i=0}^nf_ix^i$ be a symmetric polynomial, i.e., $f_i=f_{n-i}$ for any $0\leq i\leq n$. Then $f(x)$ can be expanded uniquely as
$f(x)=\sum_{k=0}^{\lrf{\frac{n}{2}}}\gamma_kx^k(1+x)^{n-2k}$, and it is said to be {\it $\gamma$-positive} if $\gamma_k\geq 0$ for $0\leq k\leq \lrf{\frac{n}{2}}$ (see~\cite{Gal05}).
The $\gamma$-positivity provides an approach to study symmetric and unimodal polynomials and has been extensively studied
(see~\cite{Athanasiadis17,Lin17,Lin15,Zhuang16} for instance).

Let $[n]=\{1,2,\ldots,n\}$.
Let $\msn$ denote the symmetric group of all permutations of $[n]$ and let $\pi=\pi(1)\pi(2)\cdots\pi(n)\in\msn$.
The number of {\it descents} of $\pi$ is defined by $$\des(\pi)=\#\{i\in[n-1]\mid \pi(i)>\pi(i+1)\}.$$
The classical {\it Eulerian number} $\Eulerian{n}{k}$ enumerates the number of permutations in $\msn$ with $k-1$ descents.
The {\it Eulerian polynomials} are defined by
$A_n(x)=\sum_{k=1}^n\Eulerian{n}{k}x^k$.
The $\gamma$-positivity of $A_n(x)$ was first studied by Foata and Sch\"utzenberger~\cite{Foata70}.
An index $i\in [n]$ is a {\it peak} (resp.~{\it double descent})
of $\pi$ if $\pi(i-1)<\pi(i)>\pi(i+1)$ (resp. $\pi(i-1)>\pi(i)>\pi(i+1)$), where $\pi(0)=\pi(n+1)=0$.
Let $a(n,k)$ be the number of permutations in $\msn$ with $k$ peaks and without double descents.
Foata and Sch\"utzenberger~\cite{Foata70} discovered that
\begin{equation}\label{Anx-gamma}
A_n(x)=\sum_{k=1}^{\lrf{({n+1})/{2}}}a(n,k)x^k(1+x)^{n+1-2k}.
\end{equation}
Moreover,
the numbers $a(n,k)$ satisfy the recurrence relation
\begin{equation*}\label{ank-recu}
a(n,k)=ka(n-1,k)+(2n-4k+4)a(n-1,k-1),
\end{equation*}
with the initial conditions $a(1,1)=1$ and $a(1,k)=0$ for $k\neq1$ (see~\cite[A101280]{Sloane}).
Subsequently, Foata and Strehl~\cite{Foata74} presented a proof of~\eqref{Anx-gamma} by
introducing a group action (which is now known as the {\it Foata-Strehl group action}) on the symmetric group $\msn$, by which
they partition $\msn$ into equivalence classes, so that for each class $\textrm{C}$,
$$\sum_{\pi\in \textrm{C}}x^{\des(\pi)}=x^i(1+x)^{n-1-2i},$$
where $i$ is the number of peaks of $\pi\in \textrm{C}$ (see~\cite{Branden08} for instance).

For an alphabet $A$, let $\mathbb{Q}((A))$ be the ring of formal Laurent series formed from letters in $A$. A {\it Chen's grammar} (also known as context-free grammar) over
$A$ is a function $G: A\rightarrow \mathbb{Q}((A))$ that replaces a letter in $A$ with an element of $\mathbb{Q}((A))$ (see~\cite{Chen93}).
The formal derivative $D$ is a linear operator defined with respect to a Chen's grammar $G$.
In other words, $D$ is the unique derivation satisfying $D(x)=G(x)$ for $x\in A$.
For example, if $A=\{x,y\}$ and $G=\{x\rightarrow xy,y\rightarrow y\}$, then $D(x)=xy,D^2(x)=D(xy)=xy^2+xy$.

Chen's grammars have been found extremely useful in deriving convolution formulas
and exponential generating functions of enumerative polynomials of
various combinatorial structures,
including set partitions~\cite{Chen93}, permutations~\cite{Chen17,Fu18,Ma131,MaYeh17,Ma1801} and increasing trees~\cite{Chen17,Dumont96}.
The purpose of the paper is to apply the grammatical method to
problems of $\gamma$-positivity.
\begin{definition}
A {\it change of grammars} is a substitution method in which the original grammars are replaced with functions of other grammars.
\end{definition}
The intent of the change of grammars method is that when expressed in new grammars, the problem may equivalent to a better understood problem.
This method can be seen in the proof of the following result.
\begin{proposition}[{\cite{Foata70}}]\label{Foata}\label{Foata}
For any $n$, the polynomial $A_n(x)$ is $\gamma$-positive.
\end{proposition}
\begin{proof}
Following~{\cite[Section~2.1]{Dumont96}},
if $A=\{x,y\}$ and
$G=\{x\rightarrow xy, y\rightarrow xy\}$,
then we have
$$D^n(x)=\sum_{k=1}^{n}\Eulerian{n}{k}x^{k}y^{n+1-k}$$ for $n\ge 1$.
Set $u=xy$ and $v=x+y$.
Note that $D(u)=uv$ and $D(v)=2u$.
It is easy to verify that
if $A=\{x,u,v\}$ and
$G=\{x\rightarrow u, u\rightarrow uv, v\rightarrow 2u\}$,
then there exist nonnegative integers $\widehat{a}(n,k)$ such that
\begin{equation}\label{Dnw02}
D^n(x)=\sum_{k=1}^{\lrf{(n+1)/2}}\widehat{a}(n,k)u^kv^{n+1-2k}.
\end{equation}
Using $D^{n+1}(x)=D(D^n(x))$, we see that numbers $\widehat{a}(n,k)$ and $a(n,k)$
satisfy the same recurrence relation and initial conditions.
Thus $\widehat{a}(n,k)=a(n,k)$.
When $y=1$, we have $u=x$ and $v=1+x$.
Then~\eqref{Dnw02} reduces to~\eqref{Anx-gamma}.
\end{proof}

We now introduce the following definition.
\begin{definition}
Let $g(x,y)$ be a bivariate polynomial of degree $n$.
If $g(x,y)$ can be expanded as
\begin{equation*}\label{gxy01}
g(x,y)=\sum_{k=0}^{\lrf{\frac{n}{2}}}\gamma_k(xy)^k(x+y)^{n-2k}
\end{equation*}
and $\gamma_k\geq 0$ for $0\leq k\leq \lrf{\frac{n}{2}}$, then we say that $g(x,y)$ is a {\it homogeneous $\gamma$-positive polynomial}.
\end{definition}
Therefore, if $g(x,y)$ is a homogeneous $\gamma$-positive polynomial, then $g(x,1)=g(1,x)$ and $g(x,1)$ is a $\gamma$-positive polynomial.

\begin{definition}
Let $p(x,y,z)$ be a three-variable polynomial. Suppose $p(x,y,z)$ can be expanded as
$p(x,y,z)=\sum_{i}s_i(x,y)z^i$.
If $s_i(x,y)$ is a homogeneous $\gamma$-positive polynomial
for every $i$, then we say that $p(x,y,z)$ is a partial $\gamma$-positive polynomial.
\end{definition}

It should be noted that partial $\gamma$-positive polynomials occur very often in combinatorics,
see~\cite{Lin15,Zeng12} for instance. In this paper, we shall introduce several partial $\gamma$-positive
polynomials.

Consider the grammar
\begin{equation}\label{type-grammar}
G=\{x\rightarrow f_1(x,y,z,\ldots), y\rightarrow f_2(x,y,z,\ldots),z\rightarrow f_3(x,y,z,\ldots),\ldots\},
\end{equation}
where $f_1(x,y,z,\ldots)=f_2(y,x,z,\ldots)$. Since $x$ and $y$ are symmetric, we say that the grammar~\eqref{type-grammar} is {\it partial symmetric}.
In this paper, the type of the change of partial symmetric grammars
is given as follows:
\begin{equation}\label{change-grammars}
\left\{
  \begin{array}{ll}
    u=xy, &  \\
    v=x+y. &
  \end{array}
\right.
\end{equation}

This paper is organized as follows. In Section~\ref{section02}, we present grammatical proofs of the $\gamma$-positivity of the
type $B$ Eulerian polynomials, derangement polynomials, Narayana polynomials and type $B$ Narayana polynomials.
From Section~\ref{section03} to Section~\ref{section06}, we provide partial
$\gamma$-positive expansions for several
polynomials associated to Stirling permutations,
Legendre-Stirling permutations, Jacobi-Stirling permutations and type $B$ derangements, and the recurrences
for the partial $\gamma$-coefficients of these expansions are also obtained.
\section{$\gamma$-positivity of several classical enumerative polynomials}\label{section02}
\subsection{Eulerian polynomials of type $B$}
\hspace*{\parindent}

Let $B_n$ be the hyperoctahedral group of rank $n$.
Elements of $B_n$ are signed permutations of $\pm[n]$ with the property that $\pi(-i)=-\pi(i)$ for all $i\in [n]$, where $\pm[n]=\{\pm1,\pm2,\ldots,\pm n\}$.
The {\it type $B$ Eulerian polynomials} are defined by
$$B_n(x)=\sum_{\pi\in B_n}x^{\des_B(\pi)},$$
where
$\des_B(\pi)=\#\{i\in\{0,1,2,\ldots,n-1\}\mid\pi(i)>\pi({i+1})\}$ and $\pi(0)=0$.
The $\gamma$-positivity of $B_n(x)$ was extensively studied by Petersen~\cite{Petersen07,Petersen15} and Chow~\cite{Chow08}.

\begin{proposition}[{\cite{Chow08,Petersen07}}]\label{Foata}\label{Petersen}
For any $n$, the polynomial $B_n(x)$ is $\gamma$-positive.
\end{proposition}
\begin{proof}
According to~\cite[Theorem~10]{Ma131},
if $A=\{x,y\}$ and
$G=\{x\rightarrow xy^2, y\rightarrow x^2y\}$,
then
$$D^n(xy)=\sum_{k=0}^nB(n,k)x^{2k+1}y^{2n-2k+1}.$$
Note that $D(xy)=xy(x^2+y^2)$ and $D(x^2+y^2)=4x^2y^2$.
Set $u=xy$ and $v=x^2+y^2$.
Then we have $D(u)=uv,D(v)=4u^2$. It is routine to check that if $A=\{u,v\}$
and
\begin{equation}\label{grammae-Euler03}
G=\{u\rightarrow uv,v\rightarrow 4u^2\},
\end{equation}
then there exist nonnegative integers $b(n,k)$ such that
$D^n(u)=u\sum_{k=0}^{\lrf{{n}/{2}}}b(n,k)u^{2k}v^{n-2k}$.
When $y=1$, we get $u=x$ and $v=1+x^2$. Hence
$B_n(x)=\sum_{k=0}^{\lrf{{n}/{2}}}b(n,k)x^k(1+x)^{n-2k}$.
\end{proof}
\subsection{Derangement polynomials}
\hspace*{\parindent}

We say that a permutation $\pi\in\msn$ is a {\it derangement} if $\pi(i)\neq i$ for every $i\in [n]$.
Let $\mdn_n$ be the set of derangements in $\msn$.
The {\it derangement polynomials} are defined by
\begin{align*}
d_n(x)=\sum_{\pi\in\mdn_n}x^{\exc(\pi)},
\end{align*}
where $\exc(\pi)=\#\{i\in [n-1]\mid \pi(i)>i\}$.
Using continued fractions, Shin and Zeng~\cite{Zeng12} studied
the $\gamma$-positivity of $d_n(x)$.

\begin{proposition}[{\cite{Zeng12}}]\label{Foata}
For any $n$, the polynomial $d_n(x)$ is $\gamma$-positive.
\end{proposition}
\begin{proof}
Following~\cite[Section~2.2]{Dumont96},
if $A=\{x,y,z,e\}$ and
$G=\{x\rightarrow xy, y\rightarrow xy,z\rightarrow xy, e\rightarrow ez\}$,
then we have
\begin{equation*}\label{Dumont}
D^n(e)=e\sum_{\pi\in\msn}x^{\exc(\pi)}y^{\dc(\pi)}z^{\fix(\pi)},
\end{equation*}
where $\fix(\pi)=\#\{i\in [n]\mid \pi(i)=i\}$ and $\dc(\pi)=\#\{i\in [n]\mid  \pi(i)<i\}$.
Set $u=xy,v=x+y$. Then we have
$D(u)=uv$ and $D(v)=2u$.
If $A=\{e,z,u,v\}$ and $G=\{e\rightarrow ez,z\rightarrow u, u\rightarrow uv, v\rightarrow 2u\}$,
then by induction, we see that there exist nonnegative integers $d(n,k)$ such that
$$D^n(e)|_{z=0}=e\sum_{k=1}^{\lrf{n/2}}d(n,k)u^kv^{n-2k}.$$
When $y=1$, we have $u=x$ and $v=1+x$. Then for $n\geq 2$, we have
$$d_n(x)=\sum_{k=1}^{\lrf{n/2}}d(n,k)x^k(1+x)^{n-2k}.$$
\end{proof}
\subsection{Narayana polynomials of types $A$ and $B$}
\hspace*{\parindent}

Let $$C_n=\frac{1}{n+1}\binom{2n}{n}$$ be the Catalan numbers.
For $1\leq k\leq n$, the {\it Narayana numbers} $N(n,k)$ are defined by
$$N(n,k)=\frac{1}{n}\binom{n}{k}\binom{n}{k-1}.$$
It is well known that $\sum_{k=0}^{n-1}N(n,k+1)x^k$ is the rank-generating function
of the lattice of non-crossing partition lattice with cardinality $C_n$ (see~\cite{Reiner97}).
There are several combinatorial interpretations of the numbers $N(n,k)$ (see~\cite{Bonin93,Chen08}).
A {\it 231-avoiding permutation} $\pi$ is a permutation with no triple of indices $i<j <k$ such that
$\pi(k)<\pi(i)<\pi(j)$.
Let $\msn(231)$ be the set of 231-avoiding permutations in $\msn$.
It is now well known (e.g. Section~2.3~\cite{Petersen15}) that
$$N(n,k)=\#\{\pi\in\msn(231)\mid \des(\pi)=k-1\}.$$

Let $N_n(x)=\sum_{k=1}^nN(n,k)x^k$ be the {\it Narayana polynomials} of type $A$.
Sulanke~\cite{Sulanke99} showed that $N_n(x)$ satisfy the recurrence relation
\begin{equation*}
(n+1)N_n(x)=(2n-1)(1+x)N_{n-1}(x)-(n-2)(x-1)^2N_{n-2}(x)
\end{equation*}
for $n\geq 2$, with the initial conditions $N_1(x)=x$ and $N_2(x)=x+x^2$.
Using generating functions and the Lagrange inversion formula,
Coker~\cite{Coker03} obtained that
\begin{equation}\label{NnxGamma}
 \sum_{k=1}^{n}N(n,k)x^k=x\sum_{k=0}^{\lrf{{(n-1)}/{2}}}\frac{1}{k+1}\binom{2k}{k}\binom{n-1}{2k}x^{k}(1+x)^{n-1-2k},
\end{equation}
Based on a weighted version of the bijection between Dyck paths and 2-Motzkin paths,
Chen, Yan and Yang~\cite{Chen08} gave a combinatorial proof of~\eqref{NnxGamma}.

For $0\leq k\leq n$, let $M(n,k)=\binom{n}{k}^2$.
The {\it Narayana polynomials of type $B$} are defined by
$M_n(x)=\sum_{k=0}^nM(n,k)x^k$.
Reiner~\cite{Reiner97} showed that $M_n(x)$ is the rank-generating function
of a ranked self-dual lattice with the cardinality $\binom{2n}{n}$.
Using generating functions, Riordan~\cite{Riordan68} derived that
\begin{equation}\label{MnxGamma}
M_n(x)=\sum_{k=0}^{\lrf{{n}/{2}}}\binom{2k}{k}\binom{n}{2k}x^k(1+x)^{n-2k}.
\end{equation}
The reader is referred to~\cite{Chen11,Petersen15,Wang18} for the recent
study on $M_n(x)$ and $N_n(x)$.
We define $$\widehat{M}(n,k)=n!M(n,k),~\widehat{N}(n,k)=(n+1)!N(n,k).$$
\begin{lemma}\label{lemmaRecu}
The numbers $\widehat{M}(n,k)$ and $\widehat{N}(n,k)$ satisfy the following recurrences:
\begin{align*}
\widehat{M}(n+1,k)&=(n+1+2k)\widehat{M}(n,k)+(3n+3-2k)\widehat{M}(n,k-1),\\
\widehat{N}(n+1,k)&=(n+2k)\widehat{N}(n,k)+(3n+4-2k)\widehat{N}(n,k-1).
\end{align*}
\end{lemma}
\begin{proof}
Using the explicit formulas for $\widehat{M}(n,k)$ and $\widehat{N}(n,k)$, we obtain
\begin{align*}
&(n+1+2k)\widehat{M}(n,k)+(3n+3-2k)\widehat{M}(n,k-1)\\
&=\frac{n!^3}{k!^2(n-k+1)!^2}(n+1)^3\\
&=(n+1)!\binom{n+1}{k}^2,\\
&(n+2k)\widehat{N}(n,k)+(3n+4-2k)\widehat{N}(n,k-1)\\
&=\frac{n!^2(n+1)!}{n}\frac{n(n+1)(n+2)}{k!(k-1)!(n-k+1)!(n-k+2)!}\\
&=\frac{(n+2)!}{n+1}\binom{n+1}{k}\binom{n+1}{k-1}.
\end{align*}
\end{proof}

\begin{lemma}\label{GrammarRecu}
If $A=\{x,y\}$ and $$G=\{x\rightarrow c_1x^{a_1+1}y^{b_1},~y\rightarrow c_2x^{a_2}y^{b_2+1}\},$$
then we have
$$D^n(x^{d_1}y^{d_2})=x^{d_1}y^{d_2}\sum_{k=0}^nc_1^kc_2^{n-k}T(n,k)x^{a_1k+a_2(n-k)}y^{b_1k+b_2(n-k)},$$
where $a_1,a_2,b_1,b_2,c_1,c_2,d_1$ and $d_2$ are given integers. Moreover, the coefficients $T(n,k)$ satisfy the recurrence relation
\begin{equation*}
T(n+1,k)=(a_2n+(a_1-a_2)(k-1)+d_1)T(n,k-1)+(b_2n+(b_1-b_2)k+d_2)T(n,k),
\end{equation*}
with the initial conditions $T(0,k)=\delta_{0,k}$, where $\delta_{i,j}$ is the Kronecker delta symbol.
\end{lemma}
\begin{proof}
Note that $D(x^{d_1}y^{d_2})=x^{d_1}y^{d_2}(c_1d_1x^{a_1}y^{b_1}+c_2d_2x^{a_2}y^{b_2})$.
Hence the result holds for $n=1$.
Suppose the statement holds for a fixed $n$. Then
we have
\begin{align*}
&D^{n+1}(x^{d_1}y^{d_2})\\
&=D(D^{n}(x^{d_1}y^{d_2}))\\
&=D\left(\sum_{k=0}^nc_1^kc_2^{n-k}T(n,k)x^{a_1k+a_2(n-k)+d_1}y^{b_1k+b_2(n-k)+d_2}\right)\\
&=x^{d_1}y^{d_2}\sum_{k=0}^n(a_1k+a_2(n-k)+d_1)c_1^{k+1}c_2^{n-k}T(n,k)x^{a_1{(k+1)}+a_2(n-k)}y^{b_1{(k+1)}+b_2(n-k)}\\
&+x^{d_1}y^{d_2}\sum_{k=0}^n(b_1k+b_2(n-k)+d_2)c_1^{k}c_2^{n-k+1}T(n,k)x^{a_1{k}+a_2(n-k+1)}y^{b_1{k}+b_2(n-k+1)}.
\end{align*}
And so the statement holds for $n+1$ and the proof is complete.
\end{proof}

The following result shows that the polynomials $M_n(x)$ and $N_n(x)$ occur in pairs.
\begin{theorem}\label{thmNarayana}
If $A=\{x,y\}$ and
\begin{equation}\label{GrammarNarayana01}
G=\{x\rightarrow x^2y^3,~y\rightarrow x^3y^2\},
\end{equation}
 then
we have
\begin{equation*}\label{GrammarNarayana02}
D^n(x^2)=(n+1)!\sum_{k=0}^nN(n,k)x^{3n-2k+2}y^{n+2k},
\end{equation*}
\begin{equation*}\label{GrammarNarayana021}
D^n(xy)=n!\sum_{k=0}^nM(n,k)x^{3n-2k+1}y^{n+2k+1}.
\end{equation*}
Moreover, the polynomials $M_n(x)$ and $N_n(x)$ are both $\gamma$-positive.
\end{theorem}
\begin{proof}
Combining Lemma~\ref{lemmaRecu} and Lemma~\ref{GrammarRecu}, we immediately get the expansions of $D^n(x^2)$ and $D^n(xy)$.
Note that $N(n,k)$ and $M(n,k)$ are both symmetric, i.e.,
$N(n,k)=N(n,n+1-k)$ and $M(n,k)=M(n,n-k)$. Then we have
\begin{equation}\label{GrammarNarayana03}
D^n(x^2)=(n+1)!(xy)^n\sum_{k=0}^nN(n,k)x^{2k}y^{2(n+1-k)},
\end{equation}
\begin{equation}\label{GrammarNarayana04}
D^n(xy)=n!(xy)^{n+1}\sum_{k=0}^nM(n,k)x^{2k}y^{2(n-k)}.
\end{equation}

Consider a change of the grammar~\eqref{GrammarNarayana01}.
Setting $w=x^2,u=xy$ and $v=x^2+y^2$, we get
$D(w)=2u^3$, $D(u)=u^2v$, $D(v)=4u^3$.
If $A=\{u,v,w\}$ and $$G=\{w\rightarrow 2u^3,~u\rightarrow u^2v,~v\rightarrow 4u^3\},$$
then by induction we see that there exist nonnegative integers $p(n,k)$ and $q(n,k)$ such that
\begin{equation*}\label{GrammarNarayana06}
D^n(w)=(n+1)!u^{n+2}\sum_{k=0}^{\lrf{(n-1)/2}}p(n,k)u^{2k}v^{n-1-2k},
\end{equation*}
\begin{equation*}\label{GrammarNarayana07}
D^n(u)=n!u^{n+1}\sum_{k=0}^{\lrf{n/2}}q(n,k)u^{2k}v^{n-2k}.
\end{equation*}
Therefore, for $n\geq 1$, we obtain
\begin{align*}
N_n(x)=x\sum_{k=0}^{(n-1)/2}p(n,k)x^k(1+x)^{n-1-2k},~
M_n(x)=\sum_{k=0}^{\lrf{n/2}}q(n,k)x^k(1+x)^{n-2k}.
\end{align*}
This completes the proof.
\end{proof}

A {\it Schr\"oder path} is a lattice path from $(0,0)$ to $(n,n)$ using the
three steps $(1,0),(0,1)$ and $(1,1)$, and not going above the line $y=x$.
Let $\operatorname{Sch}_L(n)$ be the set of all such lattice paths.
The {\it large Schr\"oder number} is defined by $r_n=\#\operatorname{Sch}_L(n)$.
For $P\in\operatorname{Sch}_L(n)$,
let $\diag(P)$ be the number of diagonal steps in the path $P$.
Let $$r_n(q)=\sum_{P\in\operatorname{Sch}_L(P)}q^{\diag(P)}.$$
According to~\cite[Section~2]{Bonin93}, we have
$r_n(q)=\sum_{k=1}^n(1+q)^kN(n,k)$.
Combining Lemma~\ref{GrammarRecu} and Theorem~\ref{thmNarayana}, we get that the polynomial
$r_n(q)$ can be generated by the following grammar: $$G=\{x\rightarrow (1+q)x^2y^3,~y\rightarrow x^3y^2\}.$$
\section{Stirling permutations}\label{section03}
\subsection{Basic definitions and notation}
\hspace*{\parindent}

The {\it Stirling numbers of the second kind} $\Stirling{n}{k}$ count the number of ways to partition
$[n]$ into $k$ non-empty subsets. Counting all functions from $[n]$ to $\{1,2,\ldots,x\}$ yields
$$x^n=\sum_{k=0}^n\Stirling{n}{k}\prod_{i=0}^{k-1}(x-i).$$
Let $C_k(x)$ be a polynomial defined by
$$\sum_{n=0}^\infty \Stirling{n+k}{n}x^n=\frac{C_k(x)}{(1-x)^{2k+1}}.$$
In~\cite{Gessel78}, Gessel and Stanley found that $C_k(x)$ is the descent polynomial of Stirling permutations of order $k$.
The first few $C_k(x)$ are given as follows:
\begin{align*}
C_1(x)=x,
C_2(x)=x+2x^2,
C_3(x)=x+8x^2+6x^3.
\end{align*}

A {\it Stirling permutation} of order $n$ is a permutation of the multiset $\{1,1,2,2,\ldots,n,n\}$ such that
for each $i$, $1\leq i\leq n$, all entries between the two occurrences of $i$ are larger than $i$.
Denote by $\mqn$ the set of {\it Stirling permutations} of order $n$.
Let $\sigma=\sigma_1\sigma_2\cdots\sigma_{2n}\in\mqn$. In the following discussion, we always set $\sigma_0=\sigma_{2n+1}=0$. For $0\leq i\leq 2n$, we say that an index $i$ is a {\it descent} (resp.~{\it ascent}, {\it plateau}) of $\sigma$ if
$\sigma_{i}>\sigma_{i+1}$ (resp. $\sigma_{i}<\sigma_{i+1}$, $\sigma_{i}=\sigma_{i+1}$).
Let $\des(\sigma),\asc(\sigma)$ and $\plat(\sigma)$ be the number of descents, ascents and plateaus of $\sigma$, respectively.
A classical result of B\'ona~\cite{Bona08} says that descents, ascents and plateaus have the same distribution over $\mqn$, i.e.,
\begin{equation}\label{Cnx}
C_n(x)=\sum_{\sigma\in\mqn}x^{\des{(\sigma)}}=\sum_{\sigma\in\mqn}x^{\asc{(\sigma)}}=\sum_{\sigma\in\mqn}x^{\plat{(\sigma)}},
\end{equation}
which has been extensively studied in~\cite{Chen17,Haglund12}.
Let $C_n(x)=\sum_{k=1}^nC(n,k)x^k$.
Very recently, Chen and Fu~\cite[Theorem~2.3]{Chen17} discovered that if $G=\{x\rightarrow xy^2, y\rightarrow xy^2\}$, then
$$D^n(x)=\sum_{k=1}^nC(n,k)x^ky^{2n+1-k}.$$

Following~\cite{Chen17}, a {\it grammatical labeling} is an assignment of the underlying elements of a combinatorial structure
with variables, which is consistent with the substitution rules of a grammar.
\subsection{Main results}
\hspace*{\parindent}

Define $$C_n(x,y,z)=\sum_{\sigma\in\mqn}x^{\asc{(\sigma)}}y^{\des(\sigma)}z^{\plat{(\sigma)}}.$$
The following lemma will be used in our discussion, which implies~\eqref{Cnx}.
\begin{lemma}[{\cite{Chen12}}]\label{Lemma-Stirling}
If $A=\{x,y,z\}$ and
\begin{equation}\label{grammar-Stirling}
G=\{x \rightarrow xyz, y\rightarrow xyz, z\rightarrow xyz\},
\end{equation}
then $D^n(x)=C_n(x,y,z)$.
\end{lemma}
\begin{proof}
We first introduce a grammatical labeling of $\sigma\in \mqn$ as follows:
\begin{itemize}
  \item [\rm ($L_1$)]If $i$ is an ascent, then put a superscript label $x$ right after $\sigma_i$;
 \item [\rm ($L_2$)]If $i$ is a descent, then put a superscript label $y$ right after $\sigma_i$;
\item [\rm ($L_3$)]If $i$ is a plateau, then put a superscript label $z$ right after $\sigma_i$.
\end{itemize}
Note that the weight of $\sigma$ is given by $w(\sigma)=x^{\asc(\sigma)}y^{\des(\sigma)}z^{\plat(\sigma)}$.
Recall that we always set $\sigma_0=\sigma_{2n+1}=0$.
Thus the index $0$ is always an ascent and the index $n$ is always a descent. We proceed by induction on $n$.
For $n=1$, we have $\mq_1=\{^x1^z1^y\}$ and $\mq_2=\{^x1^z1^x2^z2^y,^x1^x2^z2^y1^y,^x2^z2^y1^z1^y\}$.
Note that $$D(x)=xyz,D^2(x)=D(xyz)=xy^2z^2+x^2yz^2+x^2y^2z.$$
Then the weight of $^x1^z1^y$ is given by $D(x)$, and the sum of weights of the elements in $\mq_2$ is given by $D^2(x)$.
Hence the result holds for $n=1,2$.
Suppose we get all labeled permutations in $\mq_{n-1}$, where $n\geq 3$. Let
$\sigma'$ be obtained from $\sigma\in \mq_{n-1}$ by inserting the pair $nn$.
Then the changes of labeling are illustrated as follows:
$$\cdots\sigma_i^x\sigma_{i+1}\cdots\mapsto \cdots\sigma_i^xn^zn^y\sigma_{i+1}\cdots;$$
$$\cdots\sigma_i^y\sigma_{i+1}\cdots\mapsto \cdots\sigma_i^xn^zn^y\sigma_{i+1}\cdots;$$
$$\cdots\sigma_i^z\sigma_{i+1}\cdots\mapsto \cdots\sigma_i^xn^zn^y\sigma_{i+1}\cdots.$$

In each case, the insertion of $nn$ corresponds to one substitution rule in $G$.
It is easy to check that the action of $D$ on elements of $\mq_{n-1}$ generates all elements of $\mq_n$. This completes the proof.
\end{proof}

We can now present the first main result of this paper.
\begin{theorem}\label{mathm02}
For $n\geq 1$, we have
\begin{equation}\label{mathm02-112}
C_n(x,y,z)=\sum_{i=1}^{n}x^{i}\sum_{j=0}^{\lrf{(2n+1-i)/2}}\gamma_{n,i,j}(yz)^{j}(y+z)^{2n+1-i-2j}.
\end{equation}
Let $\gamma_n(x,y)=\sum_{i=1}^{n}\sum_{j\geq0}\gamma_{n,i,j}x^iy^j$. Then the polynomials $\gamma_n(x,y)$
satisfy the recurrence
\begin{equation}\label{Enxy-recu}
\gamma_{n+1}(x,y)=(4n+2)xy\gamma_n(x,y)+xy(1-2x)\frac{\partial}{\partial x}\gamma_n(x,y)+xy(1-4y)\frac{\partial}{\partial y}\gamma_n(x,y),
\end{equation}
with the initial condition $\gamma_1(x,y)=xy$.
\end{theorem}
\begin{proof}
We first consider a change of the grammar~\eqref{grammar-Stirling}. Setting $w=x,u=yz$ and $v=y+z$, we get
$D(w)=wu,D(u)=wuv,D(v)=2wu$.
If $A=\{w,u,v\}$ and
\begin{equation*}
G=\{w\rightarrow wu,u\rightarrow wuv, v\rightarrow 2wu\},
\end{equation*}
then by induction, it is routine to verify that
\begin{equation*}
D^n(w)=\sum_{i=1}^n\sum_{j=1}^{\lrf{(2n+1-i)/2}}\gamma_{n,i,j}w^iu^jv^{2n+1-i-2j}.
\end{equation*}
Then upon taking $w=x,u=yz$ and $v=y+z$, we get~\eqref{mathm02-112}.
Note that
\begin{align*}
&D^{n+1}(w)=D\left(\sum_{i,j}\gamma_{n,i,j}w^iu^jv^{2n+1-i-2j}\right)\\
&=\sum_{i,j}\gamma_{n,i,j}\left(iw^iu^{j+1}v^{2n+1-i-2j}+jw^{i+1}u^jv^{2n+2-i-2j}+2(2n+1-i-2j)w^{i+1}u^{j+1}v^{2n-i-2j}\right).
\end{align*}
Then the numbers $\gamma_{n,i,j}$ satisfy the recurrence relation
\begin{equation}\label{Enij-recu}
\gamma_{n+1,i,j}=i\gamma_{n,i,j-1}+j\gamma_{n,i-1,j}+2(2n+4-i-2j)\gamma_{n,i-1,j-1},
\end{equation}
with the initial conditions $\gamma_{1,1,1}=1$ and $\gamma_{1,i,j}=0$ for $(i,j)\neq (1,1)$.
Multiplying both sides of this recurrence relation by $x^iy^j$ and summing over all $i,j$, we get~\eqref{Enxy-recu}.
\end{proof}

The first few $\gamma_n(x,y)$ are given as follows:
\begin{align*}
\gamma_1(x,y)=xy,~
\gamma_2(x,y)=xy^2+x^2y,~
\gamma_3(x,y)=x^3y+4x^2y^2+xy^3+2x^3y^2.
\end{align*}
Recall that the Eulerian numbers $\Eulerian{n}{k}$ satisfy the recurrence relation
$$\Eulerian{n+1}{i}=i\Eulerian{n}{i}+(n+2-i)\Eulerian{n}{i-1},$$
with the initial conditions $\Eulerian{1}{1}=1$ and $\Eulerian{1}{i}=0$ for $i\neq 1$ (see~\cite{Petersen15}).
We have the following result.
\begin{proposition}
For $n\geq 1$, we have
$\gamma_{n,i,n+1-i}=\Eulerian{n}{i}$.
\end{proposition}
\begin{proof}
Set $F(n,i)=\gamma_{n,i,n+1-i}$. Then $F(n,i-1)=\gamma_{n,i-1,n+2-i}$. By using~\eqref{Enij-recu} and by induction, it is easy to verify that
$\gamma_{n,i,j}=0$ for $i+j\leq n$. Thus the numbers $F(n,i)$ satisfy the recurrence relation
$$F(n+1,i)=iF(n,i)+(n+2-i)F(n,i-1),$$
which yields the desired result.
\end{proof}
\subsection{Combinatorial interpretation of partial $\gamma$-coefficients}\label{subsection3.3}
\hspace*{\parindent}

Let $\sigma=\sigma_0\sigma_1\sigma_2\cdots \sigma_{2n}\sigma_{2n+1}\in\mqn$, where $\sigma_0=\sigma_{2n+1}=0$.
A {\it double ascent} (resp. {\it double descent}, {\it left ascent-plateau}, {\it descent-plateau}) of $\sigma$ is an index $i$ such that
$\sigma_{i-1}<\sigma_i<\sigma_{i+1}$ (resp. $\sigma_{i-1}>\sigma_i>\sigma_{i+1}$, $\sigma_{i-1}<\sigma_{i}=\sigma_{i+1}$, $\sigma_{i-1}>\sigma_i=\sigma_{i+1}$), where $i\in [2n]$.
Denote by $\dasc(\sigma)$ (resp. $\ddes(\sigma)$, $\lap(\sigma)$, $\desp(\sigma)$) the number of
double ascents (resp. double descents, left ascent-plateaus, descent-plateaus) of $\sigma$.

\begin{theorem}
For $n\geq 1$, we have
$\gamma_{n,i,j}=\#\{\sigma\in\mqn\mid \des(\sigma)=i, \lap(\sigma)=j, \desp(\sigma)=0\}$.
\end{theorem}
\begin{proof} Let $\mathcal{Q}_{n;i,j}=\{\sigma\in\mqn\mid \des(\sigma)=i, \lap(\sigma)=j, \desp(\sigma)=0\}$.
Let $\sigma\in\mathcal{Q}_{n}$. We first define two operations on $\sigma$.
For any $0\leq k\leq 2n$, let $\theta_{n+1,k}(\sigma)$ denote the element of $\mq_{n+1}$ obtained from $\sigma$ by
inserting the pair $(n+1)(n+1)$ between $\sigma_k$ and $\sigma_{k+1}$, and let $\psi_n(\sigma)$
denote the element of $\mq_{n-1}$ obtained from $\sigma$ by
deleting the pair $nn$.
We define
\begin{align*}
\Des(\sigma)&=\{k\in [2n]\mid \sigma_k>\sigma_{k+1}\},\\
\Laplat(\sigma)&=\{k\in [2n]\mid \sigma_{k-1}<\sigma_k=\sigma_{k+1}\},\\
\Dasc(\sigma)&=\{k\in [2n]\mid \sigma_{k-1}<\sigma_{k}<\sigma_{k+1}\}.
\end{align*}
For any $\sigma\in\mathcal{Q}_{n; i,j}$, we have
$|\Des(\sigma)|+2|\Laplat(\sigma)|+|\Dasc(\sigma)|=2n+1$, since $\desp(\sigma)=0.$ Thus, $|\Dasc(\sigma)|=2n+1-i-2j$.

For any $\sigma\in\mathcal{Q}_{n+1; i,j}$, denote by $r=r(\sigma)$ the index of the first occurrence of $n+1$ in $\sigma$. In other words, $\sigma_r=\sigma_{r+1}=n+1$. Then we partition the set $\mathcal{Q}_{n+1; i,j}$ into four subsets:
\begin{align*}
\mathcal{Q}_{n+1; i,j}^{1}&=\{\sigma\in\mathcal{Q}_{n+1; i,j}\mid \sigma_{r-1}>\sigma_{r+2}\};\\
\mathcal{Q}_{n+1; i,j}^{2}&=\{\sigma\in\mathcal{Q}_{n+1; i,j}\mid \sigma_{r-2}<\sigma_{r-1}=\sigma_{r+2}\};\\
\mathcal{Q}_{n+1; i,j}^{3}&=\{\sigma\in\mathcal{Q}_{n+1; i,j}\mid \sigma_{r-1}<\sigma_{r+2}<\sigma_{r+3}\};\\
\mathcal{Q}_{n+1; i,j}^{4}&=\{\sigma\in\mathcal{Q}_{n+1; i,j}\mid \sigma_{r-2}>\sigma_{r-1}=\sigma_{r+2}\}.
\end{align*}

{\it Claim 1.} There is a bijection $\phi_1:\mathcal{Q}_{n+1; i,j}^{1}\mapsto \{(\pi,k)\mid
\pi\in\mathcal{Q}_{n; i,j-1}\text{ and }k\in \Des(\pi)\}$.

For any $\sigma\in\mathcal{Q}_{n+1; i,j}^{1}$, note that $\psi_{n+1}(\sigma)\in \mathcal{Q}_{n; i,j-1}$ and
$r(\sigma)-1\in \Des(\psi_{n+1}(\sigma))$. We define
the map $\phi_1:\mathcal{Q}_{n+1; i,j}^{1}\mapsto\{(\pi,k)\mid\pi\in\mathcal{Q}_{n; i,j-1}\text{ and }k\in \Des(\pi)\}$ by letting $$\phi_1(\sigma)=(\psi_{n+1}(\sigma),r(\sigma)-1).$$
The inverse of $\phi_1^{-1}$ is given by $\phi_1^{-1}(\pi,k)=\theta_{n+1,k}(\pi)$.

{\it Claim 2.} There is a bijection $\phi_2:\mathcal{Q}_{n+1; i,j}^{2}\mapsto
\{(\pi,k)\mid \pi\in\mathcal{Q}_{n; i-1,j}\text{ and }k\in \Laplat(\pi)\}$.

For any $\sigma\in\mathcal{Q}_{n+1; i,j}^{2}$, note that
$\psi_{n+1}(\sigma)\in \mathcal{Q}_{n; i-1,j}$ and $r(\sigma)-1\in \Laplat(\psi_{n+1}(\sigma))$.
We define the map $\phi_2:\mathcal{Q}_{n+1; i,j}^{2}\mapsto\{(\pi,k)\mid \pi\in\mathcal{Q}_{n; i-1,j}\text{ and }k\in \Laplat(\pi)\}$ by letting $$\phi_2(\sigma)=(\psi_{n+1}(\sigma),r(\sigma)-1).$$
The inverse of $\phi_2^{-1}$ is given by $\phi_2^{-1}(\pi,k)=\theta_{n+1,k}(\pi)$.

{\it Claim 3.} There is a bijection $\phi_3:\mathcal{Q}_{n+1; i,j}^{3}\mapsto
\{(\pi,k)\mid \pi\in\mathcal{Q}_{n; i-1,j-1}\text{ and }k\in \Dasc(\pi)\}$.

For any $\sigma\in\mathcal{Q}_{n+1; i,j}^{3}$, note that
$\psi_{n+1}(\sigma)\in \mathcal{Q}_{n; i-1,j-1}$ and $r(\sigma)\in \Dasc(\psi_{n+1}(\sigma))$.
We define the map $\phi_3:\mathcal{Q}_{n+1; i,j}^{3}\mapsto\{(\pi,k)\mid \pi\in\mathcal{Q}_{n; i-1,j-1}\text{ and }k\in \Dasc(\pi)\}$ by letting $$\phi_3(\sigma)=(\psi_{n+1}(\sigma),r(\sigma)).$$ The inverse of $\phi_3^{-1}$
is given by $\phi_3^{-1}(\pi,k)=\theta_{n+1,k-1}(\pi)$.

{\it Claim 4.} There is a bijection $\phi_4:\mathcal{Q}_{n+1; i,j}^{4}\mapsto
\{(\pi,k)\mid \pi\in\mathcal{Q}_{n; i-1,j-1}\text{ and }k\in \Dasc(\pi)\}$.

Let $k \in [2n]$ and let $\sigma\in\mathcal{Q}_n$.
We define a {\it modified Foata-Strehl group action} $\varphi_k$ as follows:
\begin{itemize}\item  If $k$ is a double ascent then $\varphi_k(\sigma)$ is obtained by moving $\sigma_k$ to the left of the second $\sigma_k$, which forms a
new plateau $\sigma_k\sigma_k$;
\item If $k$ is a descent-plateau then $\varphi_k(\sigma)$ is obtained by moving $\sigma_k$ to the right of $\sigma_{j}$, where
$j=\max\{s\in \{0,1,2,\ldots,k-1\}:\sigma_s < \sigma_k\}$.
\end{itemize}

For instance, if $\sigma=2447887332115665$, then $$\varphi_{1}(\sigma)=4478873322115665,~\varphi_{4}(\sigma)=2448877332115665,$$
$\varphi_{9}\circ\varphi_{1}(\sigma)=\sigma$ and $\varphi_{6}\circ\varphi_{4}(\sigma)=\sigma$, where $\circ$ denotes the composition operation.

For any $\sigma\in\mathcal{Q}_{n+1; i, j}^{4}$, note that the index $r(\sigma)-1$ is the unique descent-plateau in $\psi_{n+1}(\sigma)$ and $\varphi_{r(\sigma)-1}\circ \psi_{n+1}(\sigma)\in \mathcal{Q}_{n; i-1, j-1}$. Let $\sigma'=\varphi_{r(\sigma)-1}\circ \psi_{n+1}(\sigma)$. Read $\sigma'$ from left to right and let $p$ be the index of the first occurrence of the integer $\sigma_{r(\sigma)-1}$.
It is clear that $$p\in \Dasc(\varphi_{r(\sigma)-1}\circ \psi_{n+1}(\sigma)).$$
Therefore, we define the map $\phi_4:\mathcal{Q}_{n+1; i, j}^{4}\mapsto\{(\pi,k)\mid \pi\in\mathcal{Q}_{n; i-1,j-1}\text{ and }k\in \Dasc(\pi)\}$ by letting $\phi_4(\sigma)=(\varphi_{r(\sigma)-1}\circ \psi_{n+1}(\sigma), p)$.
For any $\pi\in\mathcal{Q}_{n; i-1, j-1}\text{ and }k\in \Dasc(\pi)$,
the inverse of $\phi_4^{-1}$ is given by $\phi_4^{-1}(\pi,k)=\theta_{n+1,r}(\varphi_{k}(\pi))$, where $r$ is the unique descent-plateau in $\varphi_{k}(\pi)$. Thus $\phi_4$ is the desired bijection.

In conclusion, we have
\begin{align*}
|\mathcal{Q}_{n+1; i,j}|&=|\mathcal{Q}_{n+1; i,j}^{1}|+|\mathcal{Q}_{n+1; i,j}^{2}|+|\mathcal{Q}_{n+1; i,j}^{3}|+|\mathcal{Q}_{n+1; i,j}^{4}|\\
&=i|\mathcal{Q}_{n; i,j-1}|+j|\mathcal{Q}_{n; i-1,j}|+2(2n+4-i-2j)|\mathcal{Q}_{n; i-1,j-1}|.
\end{align*}

It is clear that $\mathcal{Q}_{1;1,1}=\{11\}$ and $\mathcal{Q}_{1;i,j}=\emptyset$ for any $(i,j)\neq (1,1)$. So, $\gamma_{1;1,1}=1=|\mathcal{Q}_{1;1,1}|$ and $\gamma_{1;i,j}=0=|\mathcal{Q}_{1;i,j}|$ for any $(i,j)\neq (1,1).$
By induction, we get $|\mathcal{Q}_{n+1;i,j}|=\gamma_{n+1,i,j}$ and this completes the proof.
\end{proof}
\section{Legendre-Stirling permutations}\label{section04}
\subsection{Basic definitions and notation}
\hspace*{\parindent}

The Legendre-Stirling numbers of the second kind $\LS(n,k)$ first arose in the study of
a certain differential operator related to
Legendre polynomials (see~\cite{Everitt02}).
The numbers $\LS(n,k)$ can be defined as follows:
$$x^n=\sum_{j=0}^n\LS(n,k)\prod_{i=0}^{k-1}(x-i(i+1)),$$
and satisfy the recurrence relation
$\LS(n,k)=\LS(n-1,k-1)+k(k+1)\LS(n-1,k)$,
with the initial conditions $\LS(0,0)=1$ and $\LS(0,k)=0$ for $k\geq 1$.
Andrews and Littlejohn~\cite{Andrews09} discovered that $\LS(n,k)$ is the number of Legendre-Stirling set partitions of
the set $\{1_1,1_2,2_1,2_2,\ldots,n_1,n_2\}$ into $k$ nonzero blocks and one zero block.
Subsequently, Egge~\cite{Egge10} considered
the polynomial $L_k(x)$, which is defined by
$$\sum_{n=0}^\infty \LS(n+k,n)x^n=\frac{L_k(x)}{(1-x)^{3k+1}},$$
and found that $L_k(x)$ is the descent polynomial of Legendre-Stirling permutations of order $k$.
The reader is referred to~\cite{Andrews11,Egge10} for further properties of the Legendre-Stirling numbers.

For $n\geq 1$, let $N_n$ denote the multiset $\{1,1,\overline{1},2,2,\overline{2},\ldots,n,n,\overline{n}\}$,
in which we have two unbarred copies and one barred copy of each integer $i$, where $1\leq i\leq n$.
In this section, we always assume that the elements of $N_n$ are ordered by
$\overline{1}=1<\overline{2}=2<\cdots <\overline{n}=n$.
Here the $\overline{k}=k$ means that $\overline{k}k$ counts as a plateau.

A {\it Legendre-Stirling permutation} of order $n$ is a permutation of $N_n$ such that
if $i<j<k$, $\pi_i$ and $\pi_k$ are both unbarred and $\pi_i=\pi_k$, then $\pi_j>\pi_i$.
Let $\operatorname{LS}_n$ denote the set of Legendre-Stirling permutations of order $n$. Let $\pi=\pi_1\pi_2\cdots\pi_{3n}\in\operatorname{LS}_n$. We always set $\pi_0=\pi_{3n+1}=0$.
An index $i$ is a {\it descent} (resp. {\it ascent}, {\it plateau}) of $\pi$ if
$\pi_i>\pi_{i+1}$ (resp. $\pi_i<\pi_{i+1}$, $\pi_i=\pi_{i+1}$).
Hence the index $i=0$ is always an ascent and $i=3n$ is always a descent.
Denote by $\des(\pi)$ (resp.~$\asc(\pi)$, $\plat(\pi)$) the number
of descents (resp.~ascents, plateaus) of $\pi$.
Let $$L_n(x)=\sum_{\pi\in\operatorname{LS}_n}x^{\des(\pi)}.$$
The first few $L_n(x)$ are given as follows:
\begin{align*}
L_1(x)=2x,
L_2(x)=4x+24x^2+12x^3,
L_3(x)=8x+240x^2+984x^3+864x^4+144x^5.
\end{align*}

Let $\operatorname{LSD}_n$ be the set of Legendre-Stirling permutations of the multiset $ND_n=N_n\setminus\{n,n\}$, where a Legendre-Stirling permutation
of $ND_n$ is a permutation of $ND_n$ such that
if $i<j<k$, $\pi_i$ and $\pi_k$ are both unbarred and $\pi_i=\pi_k$, then $\pi_j>\pi_i$. For $\pi=\pi_1\pi_2\cdots\pi_{3n-2}\in\operatorname{LSD}_n$,
we always set $\pi_0=\pi_{3n-1}=0$.
\subsection{Main results}
\hspace*{\parindent}

For $n\geq 1$, we define
\begin{align*}
H_n(x,y,z)&=\sum_{\pi\in\operatorname{LSD}_n}x^{\asc(\pi)-1}y^{\des(\pi)-1}z^{\plat(\pi)},\\
L_n(x,y,z)&=\sum_{\pi\in\operatorname{LS}_n}x^{\asc(\pi)}y^{\des(\pi)}z^{\plat(\pi)}.
\end{align*}

\begin{lemma}\label{lemma-Leg}
Let $A=\{u,v,x,y,z\}$ and
\begin{equation}\label{Leg01}
G_1=\{x\rightarrow uv,y\rightarrow uv,z\rightarrow uv\},
\end{equation}
\begin{equation}\label{Leg02}
~G_2=\{x\rightarrow \frac{x^2y^2z}{uv},y\rightarrow \frac{x^2y^2z}{uv},z\rightarrow \frac{x^2y^2z}{uv},
u\rightarrow \frac{xyz^2}{v},v\rightarrow \frac{xyz^2}{u}\}.
\end{equation}
Then for $n\geq 1$, we have
\begin{align*}
D_1(D_2D_1)^{n-1}(x)=uvH_n(x,y,z),~(D_2D_1)^n(x)=L_n(x,y,z),
\end{align*}
where $D_2D_1$ is a composition operator, i.e., $(D_2D_1)^n=D_2\left(D_1(D_2D_1)^{n-1}\right)$.
\end{lemma}
\begin{proof}
Note that every permutation in $\operatorname{LS}_n$ can be obtained from a permutation in $\operatorname{LS}_{n-1}$ by first inserting $\overline{n}$ between two entries, and then inserting the pair $nn$ between two entries of this new permutation.
We first introduce a grammatical labeling of $\pi\in \operatorname{LSD}_n$ as follows:
\begin{itemize}
  \item [\rm ($L_1$)] Put a superscript label $u$ immediately before the entry $\overline{n}$ and a superscript label $v$ right after $\overline{n}$;
  \item [\rm ($L_2$)]If $i$ is an ascent and $\pi_{i+1}\neq \overline{n}$, then put a superscript label $x$ right after $\pi_i$;
 \item [\rm ($L_3$)]If $i$ is a descent and $\pi_{i}\neq \overline{n}$, then put a superscript label $y$ right after $\pi_i$;
\item [\rm ($L_4$)]If $i$ is a plateau, then put a superscript label $z$ right after $\pi_i$.
\end{itemize}
Thus, the weight of $\pi\in\operatorname{LSD}_n$ is given by $w(\pi)=uvx^{\asc(\pi)-1}y^{\des(\pi)-1}z^{\plat(\pi)}$.
We then introduce a grammatical labeling of $\pi\in \operatorname{LS}_n$ as follows:
\begin{itemize}
  \item [\rm ($L_1$)]If $i$ is an ascent, then put a superscript label $x$ right after $\pi_i$;
 \item [\rm ($L_2$)]If $i$ is a descent, then put a superscript label $y$ right after $\pi_i$;
\item [\rm ($L_3$)]If $i$ is a plateau, then put a superscript label $z$ right after $\pi_i$.
\end{itemize}
Thus, the weight of $\pi\in\LS_n$ is given by $w(\pi)=x^{\asc(\pi)}y^{\des(\pi)}z^{\plat(\pi)}$.

We proceed by induction on $n$. When $n=1$, we have
$\operatorname{LSD}_1=\{^u\overline{1}^v\},~\operatorname{LS}_1=\{^x\overline{1}^z1^z1^y,^x1^z1^z\overline{1}^y\}$.
Note that $D_1(x)=uv,(D_2D_1)(x)=D_2(D_1(x))=D_2(uv)=2xyz^2$.
Then the weight of $\overline{1}$ is given by $D_1(x)$, and the sum of weights of the elements in $\operatorname{LS}_1$ is given by $(D_2D_1)(x)$.
Hence the results hold for $n=1$.
Suppose we get all labeled permutations in $\operatorname{LS}_{n-1}$, where $n\geq 2$.
Let
$\pi'$ be obtained from $\pi\in \operatorname{LS}_{n-1}$ by inserting the entry $\overline{n}$ to a position with a label $x,y$ or $z$. The changes of labeling are  illustrated as follows: $$\cdots\pi_i^x\pi_{i+1}\cdots\mapsto \cdots\pi_i^u\overline{n}^v\pi_{i+1}\cdots;$$
$$\cdots\pi_i^y\pi_{i+1}\cdots\mapsto \cdots\pi_i^u\overline{n}^v\pi_{i+1}\cdots;$$
$$\cdots\pi_i^z\pi_{i+1}\cdots\mapsto \cdots\pi_i^u\overline{n}^v\pi_{i+1}\cdots.$$
In each case, the insertion of $\overline{n}$ corresponds to one substitution rule in $G_1$.
For $\pi\in \operatorname{LSD}_{n}$, we now insert the pair $nn$ into $\pi$. We distinguish the following cases:
\begin{itemize}
  \item [\rm ($c_1$)] If $nn$ is inserted at a position with the label $x,y$ or $z$, then the changes of labeling can be illustrated as follows: $$\cdots^u\overline{n}^v\cdots\pi_i^x\pi_{i+1}\cdots\mapsto \cdots^x\overline{n}^y\cdots\pi_i^xn^zn^y\pi_{i+1}\cdots;$$
       $$\cdots^u\overline{n}^v\cdots\pi_i^y\pi_{i+1}\cdots\mapsto \cdots^x\overline{n}^y\cdots\pi_i^xn^zn^y\pi_{i+1}\cdots;$$
      $$\cdots^u\overline{n}^v\cdots\pi_i^z\pi_{i+1}\cdots\mapsto \cdots^x\overline{n}^y\cdots\pi_i^xn^zn^y\pi_{i+1}\cdots.$$
\item [\rm ($c_2$)] If $nn$ is inserted at a position with the label $u$, then the change of labeling is illustrated as follows:
$\cdots^u\overline{n}^v\cdots\mapsto \cdots^xn^zn^z\overline{n}^y\cdots$.
\item [\rm ($c_2$)] If $nn$ is inserted at a position with the label $v$, then the change of labeling is illustrated as follows:
$\cdots^u\overline{n}^v\cdots\mapsto \cdots^x\overline{n}^zn^zn^y\cdots$.
\end{itemize}
In each case, the insertion of the pair $nn$ corresponds to one substitution rule in $G_2$.
It is routine to check that the action of $D_2D_1$ on Legendre-Stirling permutations of $\operatorname{LS}_{n-1}$
generates all the Legendre-Stirling permutations in $\operatorname{LS}_n$. This completes the proof.
\end{proof}

We can now present the second main result of this paper.
\begin{theorem}\label{thm-LS}
For $n\geq 1$, we have
\begin{equation*}\label{mathm02-11}
H_n(x,y,z)=\sum_{i=1}^{2n-2}z^{i}\sum_{j=0}^{\lrf{(3n-3-i)/2}}h_{n}(i,j)(xy)^{j}(x+y)^{3n-3-i-2j},
\end{equation*}
\begin{equation*}\label{mathm02-12}
L_n(x,y,z)=\sum_{i=1}^{2n}z^{i}\sum_{j=1}^{\lrf{(3n+1-i)/2}}\ell_{n}(i,j)(xy)^{j}(x+y)^{3n+1-i-2j},
\end{equation*}
where the numbers $h_n(i,j)$ and $\ell_n(i,j)$ satisfy the recurrence relations
\begin{align*}
\ell_n(i,j)&=2h_n(i-2,j-1)+ih_n(i,j-2)+(j-1)h_n(i-1,j-1)+\\
&2(3n+2-i-2j)h_n(i-1,j-2),\\
h_{n+1}(i.j)&=(i+1)\ell_n(i+1,j)+(j+1)\ell_n(i,j+1)+2(3n+1-i-2j)\ell_n(i,j),
\end{align*}
with the initial conditions $h_1(0,0)=1$ and $h_1(i,j)=0$ for $(i,j)\neq (0,0)$, $\ell_1(2,1)=2$ and $\ell_1(i,j)=0$ for $(i,j)\neq (2,1)$.
\end{theorem}
\begin{proof}
Consider the grammars~\eqref{Leg01} and~\eqref{Leg02}.
Setting $a=x+y, b=xy$, we get
$$D_1(a)=2uv,D_1(b)=auv, D_2(x)=\frac{zb^2}{uv},D_2(y)=\frac{zb^2}{uv},D_2(z)=\frac{zb^2}{uv},$$
$$D_2(u)=\frac{z^2b}{v},D_2(v)=\frac{z^2b}{u},D_2(a)=\frac{2zb^2}{uv},D_2(b)=\frac{zab^2}{uv}.$$
Then the changes of grammars are given as follows: $A=\{a,b,x,y,z,u,v\}$ and
\begin{equation}\label{Leg-gram01}
G_3=\{x\rightarrow uv,z\rightarrow uv,a\rightarrow 2uv, b\rightarrow auv\},
\end{equation}
\begin{equation}\label{Leg-gram02}
G_4=\{x\rightarrow\frac{zb^2}{uv},y\rightarrow\frac{zb^2}{uv},z\rightarrow\frac{zb^2}{uv},u\rightarrow \frac{z^2b}{v},v\rightarrow \frac{z^2b}{u},a\rightarrow \frac{2zb^2}{uv},b\rightarrow \frac{zab^2}{uv}\}.
\end{equation}
Note that
$D_3(x)=uv,~D_4D_3(x)=2z^2b,~D_3(2z^2b)=4zbuv+2z^2auv$.
By induction, it is routine to verify that there exist nonnegative integers $h_n(i,j)$ and $\ell_n(i,j)$ such that
\begin{align*}
D_3(D_4D_3)^{n-1}(x)&=uv\sum_{i=1}^{2n-2}z^{i}\sum_{j=0}^{\lrf{(3n-3-i)/2}}h_{n}(i,j)b^{j}a^{3n-3-i-2j},\\
(D_4D_3)^{n}(x)&=\sum_{i=1}^{2n}z^{i}\sum_{j=1}^{\lrf{(3n+1-i)/2}}\ell_{n}(i,j)b^{j}a^{3n+1-i-2j}.
\end{align*}
Then upon taking $a=x+y$ and $b=xy$, it follows from Lemma~\ref{lemma-Leg} that we get the expansions of $H_n(x,y,z)$ and $L_n(x,y,z)$.
From
\begin{align*}
D_4(D_3(D_4D_3)^{n-1}(x))&=D_4\left(\sum_{i,j}h_n(i,j)uvz^ib^ja^{3n-3-i-2j}\right)\\
&=\sum_{i,j}h_n(i,j)(2z^{i+2}b^{j+1}a^{3n-3-i-2j}+iz^ib^{j+2}a^{3n-3-i-2j})+\\
&\sum_{i,j}h_n(i,j)(jz^{i+1}b^{j+1}a^{3n-2-i-2j}+2(3n-3-i-2j)z^{i+1}b^{j+2}a^{3n-4-i-2j}),
\end{align*}
and
\begin{align*}
D_3((D_4D_3)^{n}(x))&=D_3\left(\sum_{i,j}\ell_n(i,j)z^ib^ja^{3n+1-i-2j}\right)\\
&=uv\sum_{i,j}\ell_n(i,j)iz^{i-1}b^ja^{3n+1-i-2j}+uv\sum_{i,j}\ell_n(i,j)jz^ib^{j-1}a^{3n+2-i-2j}+\\
&uv\sum_{i,j}\ell_n(i,j)2(3n+1-i-2j)z^ib^ja^{3n-i-2j},
\end{align*}
we get the desired recurrence relations.
In particular, $h_1(0,0)=1,h_2(1,1)=4,h_2(2,0)=2$ and $\ell_1(2,1)=2$.
This completes the proof.
\end{proof}

Using~\eqref{Leg-gram01} and~\eqref{Leg-gram02}, it is not hard to verify that
$\ell_n(i,j)$ satisfy the recurrence relation
\begin{align*}
\ell_{n+1}(i,j)&=i(i+1)\ell_n(i+1,j-2)+2i(j-1)\ell_n(i,j-1)+j(j-1)\ell_n(i-1,j)+\\
&2j\ell_n(i-2,j)+4i(3n+5-i-2j)\ell_n(i,j-2)+4(3n+5-i-2j)\ell(i-2,j-1)+\\
&4(3n+6-i-2j)(3n+5-i-2j)\ell_n(i-1,j-2)+\\
&2((2j-2)(3n+4-i-2j)+i+j-2)\ell_n(i-1,j-1),
\end{align*}
with the initial conditions $\ell_1(2,1)=2$ and $\ell_1(i,j)=0$ for $(i,j)\neq (2,1)$.

We define
$$\overline{H}_n(x,y)=\sum_{i=1}^{2n-2}\sum_{j=0}^{\lrf{(3n-3-i)/2}}h_{n}(i,j)x^iy^j,~
\overline{L}_n(x,y)=\sum_{i=1}^{2n}\sum_{j=1}^{\lrf{(3n+1-i)/2}}\ell_{n}(i,j)x^iy^j.$$
Using Theorem~\ref{thm-LS},
multiplying both sides of the recurrence relations of $h_n(i,j)$ and $\ell_n(i,j)$ by $x^iy^j$ and summing over all $i,j$, we get the following recurrence relations:
\begin{align*}
\overline{L}_n(x,y)&=xy(6ny-6y+2x)\overline{H}_n(x,y)+xy^2(1-2x)\frac{\partial}{\partial x}\overline{H}_n(x,y)+xy^2(1-4y)\frac{\partial}{\partial y}\overline{H}_n(x,y),\\
\overline{H}_{n+1}(x,y)&=(6n+2)\overline{L}_{n}(x,y)+(1-2x)\frac{\partial}{\partial x}\overline{L}_{n}(x,y)+(1-4y)\frac{\partial}{\partial y}\overline{L}_{n}(x,y).
\end{align*}
The first few $\overline{H}_n(x,y)$ and $\overline{L}_n(x,y)$ are given as follows:
\begin{align*}
\overline{H}_1(x,y)=1,~
\overline{L}_1(x,y)=2x^2y,~
\overline{H}_2(x,y)=4xy+2x^2,~
\overline{L}_2(x,y)=4xy^3+8x^2y^2+12x^3y^2+4x^4y.
\end{align*}
\section{Jacobi-Stirling permutations}\label{section05}
\subsection{Definitions and notation}
\hspace*{\parindent}

The Jacobi-Stirling numbers $\JS(n,k;z)$ were discovered as a result of a problem involving the spectral theory
of powers of the classical second-order Jacobi differential expression (see~\cite{Andrews09,Everitt07}), and they can be defined as follows:
$$x^n=\sum_{k=0}^n\JS(n,k;z)\prod_{i=0}^{k-1}(x-i(z+i)).$$
In particular, $\JS(n,k;1)=\LS(n,k)$. The reader is referred to~\cite{Andrews13,Zeng10}
for further properties of the Jacobi-Stirling numbers. The Jacobi-Stirling polynomial of the second kind is defined by
$f_k(n;z)=\JS(n+k,n;z)$.
The coefficient $p_{k,i}(n)$ of $z^i$ in $f_k(n;z)$ is called the Jacobi-Stirling coefficient of the second kind for $0\leq i\leq k$.
Gessel, Lin and Zeng~\cite{Gessel12} found a combinatorial interpretation of
the polynomial $A_{k,i}(x)$ which is defined by
$$\sum_{n\geq 0}p_{k,i}(n)x^n=\frac{A_{k,i}(x)}{(1-x)^{3k-i+1}}.$$

We define the multiset
$M_k=\{\overline{1},1,1,\overline{2},2,2,\ldots,\overline{k},k,k\}$,
in which we have two unbarred copies and one barred copy of each integer $i$, where $1\leq i\leq k$.
In this section, we always assume that the elements of $M_k$ are ordered by
$$\overline{1}<1<\overline{2}<2<\cdots <\overline{k}<k.$$
A permutation of $M_{k}$ is a {\it Jacobi-Stirling permutation} if
for each $i$, $1\leq i\leq k$, all entries between the two occurrences of the unbarred $i$ are larger than $i$.
Let $\operatorname{JSP}_{k}$ denote the set of Jacobi-Stirling permutations of $M_{k}$. For example,
$\operatorname{JSP}_1=\{\overline{1}11,11\overline{1}\}$. For $\pi=\pi_1\pi_2\cdots\pi_{3k}\in \operatorname{JSP}_k$,
we always set $\pi_0=\pi_{3k+1}=0$.
We define
\begin{align*}
\des(\pi)&=\#\{i\in [3n]\mid \pi_i>\pi_{i+1}\},\\
\asc(\pi)&=\#\{i\in \{0,1,2,\ldots,3n-1\}\mid \pi_i<\pi_{i+1}\},\\
\plat(\pi)&=\#\{i\in [3n-1]\mid \pi_i=\pi_{i+1}\}.
\end{align*}
It follows from~\cite[Theorem~2]{Gessel12} that
$$(1-x)^{3k+1}\sum_{n\geq 0}p_{k,0}(n)x^n=\sum_{\pi\in\JSP_{k}}x^{\des(\pi)}.$$

\subsection{Main results}
\hspace*{\parindent}

Define
\begin{equation*}\label{mathm02-1}
S_n(x,y,z)=\sum_{\pi\in \operatorname{JSP}_n}x^{\asc(\pi)}y^{\des(\pi)}z^{\plat(\pi)}.
\end{equation*}
The first few $S_n(x,y,z)$ are given as follows:
\begin{align*}
S_1(x,y,z)&=xy(x+y)z,\\
S_2(x,y,z)&=(xy)^2(3x^2+10xy+3y^2)z+xy(x^3+11x^2y+11xy^2+y^3)z^2,\\
S_3(x,y,z)&=(xy)^3(17x^3+119x^2y+119xy^2+17y^3)z+\\
&(xy)^2(18x^4+284x^3y+644x^2y^2+284xy^2+18y^4)z^2+\\
&(xy)(x^5+57x^4y+302x^3y^2+302x^2y^3+57xy^4+y^5)z^3.
\end{align*}

\begin{lemma}\label{Lemma-Jacobi}
Let $A=\{x,y,z\}$ and
\begin{equation}\label{Grammar-Jacobi}
G_1=\{x\rightarrow xy, y\rightarrow xy, z\rightarrow xy\},~ G_2=\{x\rightarrow xyz, y\rightarrow xyz, z\rightarrow xyz\}.
\end{equation}
Then for $n\geq 1$, we have
$(D_2D_1)^n(x)=(D_2D_1)^n(y)=(D_2D_1)^n(z)=S_n(x,y,z)$.
\end{lemma}
\begin{proof}
We first introduce a grammatical labeling of $\pi\in\operatorname{JSP}_n$ as follows:
\begin{itemize}
  \item [\rm ($L_1$)]If $i$ is an ascent, then put a superscript label $x$ right after $\pi_i$;
 \item [\rm ($L_2$)]If $i$ is a descent, then put a superscript label $y$ right after $\pi_i$;
\item [\rm ($L_3$)]If $i$ is a plateau, then put a superscript label $z$ right after $\pi_i$.
\end{itemize}
Thus the weight of $\pi$ is given by $w(\pi)=x^{\asc(\pi)}y^{\des(\pi)}z^{\plat(\pi)}$.
We proceed by induction on $n$.
For $n=1$, we have $\operatorname{JSP}_1=\{^x\overline{1}^x1^z1^y,~^x1^z1^y\overline{1}^y\}$.
Note that $$(D_2D_1)(x)=(D_2D_1)(y)=(D_2D_1)(z)=xy^2z+x^2yz.$$
Hence the result holds for $n=1$.
Note that any permutation of $\operatorname{JSP}_n$ is obtained from a permutation of $\operatorname{JSP}_{n-1}$
by first inserting the element $\overline{n}$ and then inserting the pair $nn$.

We first insert $\overline{n}$ and the changes of labeling are illustrated as follows:
 $$\cdots\pi_i^x\pi_{i+1}\cdots\mapsto \cdots\pi_i^x\overline{n}^y\pi_{i+1}\cdots;$$
 $$\cdots\pi_i^y\pi_{i+1}\cdots\mapsto \cdots\pi_i^x\overline{n}^y\pi_{i+1}\cdots;$$
 $$\cdots\pi_i^z\pi_{i+1}\cdots\mapsto \cdots\pi_i^x\overline{n}^y\pi_{i+1}\cdots.$$
In each case, the insertion of $\overline{n}$ corresponds to one substitution rule in $G_1$.
We then insert the pair $nn$ and the changes of labeling are illustrated as follows:
$$\cdots\pi_i^x\pi_{i+1}\cdots\mapsto \cdots\pi_i^xn^zn^y\pi_{i+1}\cdots;$$
$$\cdots\pi_i^y\pi_{i+1}\cdots\mapsto \cdots\pi_i^xn^zn^y\pi_{i+1}\cdots;$$
$$\cdots\pi_i^z\pi_{i+1}\cdots\mapsto \cdots\pi_i^xn^zn^y\pi_{i+1}\cdots.$$
In each case, the insertion of the pair $nn$ corresponds to one substitution rule in $G_2$.
It is easy to check that the action of $D_2D_1$ on elements of $\operatorname{JSP}_{n-1}$ generates all the elements in $\operatorname{JSP}_n$. This completes the proof.
\end{proof}

We can now present the third main result of this paper.
\begin{theorem}\label{mathm02}
For $n\geq 1$, we have
\begin{equation}\label{mathm02-13}
S_n(x,y,z)=\sum_{i=1}^{n}z^{i}\sum_{j=1}^{\lrf{(3n+1-i)/2}}s_{n}(i,j)(xy)^{j}(x+y)^{3n+1-i-2j},
\end{equation}
where the numbers $s_n(i,j)$ satisfy the recurrence relation
\begin{align*}
s_{n+1}(i,j)&=i(i+1)s_n(i+1,j-2)+i(2j-1)s_n(i,j-1)+4i(3n+5-i-2j)s_n(i,j-2)+\\
&j^2s_n(i-1,j)+(4(j-1)(3n+4-i-2j)+6n+6-2i-2j)s_n(i-1,j-1)+\\
&4(3n+6-i-2j)(3n+5-i-2j)s_n(i-1,j-2),
\end{align*}
with the initial conditions $s_0(1,0)=1$ and $s_0(i,j)=0$ for $(i,j)\neq (1,0)$.
\end{theorem}
\begin{proof}
Consider the grammars~\eqref{Grammar-Jacobi}.
Setting
$a=z,b=x+y,c=xy$,
we get
$D_1(a)=c,D_1(b)=2c,D_1(c)=bc$ and
$D_2(a)=ac,D_2(b)=2ac,D_2(c)=abc$.
Then the changes of grammars are given as follows: $A=\{a,b,c\}$ and
\begin{equation}\label{JAcobi-gram02}
G_3=\{a\rightarrow c,b\rightarrow 2c, c\rightarrow bc\},~G_4=\{a\rightarrow ac,b\rightarrow 2ac, c\rightarrow abc\}.
\end{equation}
Note that $D_4D_3(a)=D_4(c)=abc$ and $(D_4D_3)^2(a)=a(3b^2c^2+4c^3)+a^2(b^3c+8bc^2)$.
In general, there exist nonnegative integers $r_n(i,j,k)$ such that
\begin{equation}\label{D4D3a01}
(D_4D_3)^n(a)=\sum_{i,j,k}r_n(i,j,k)a^ic^jb^k.
\end{equation}
We use the following equivalent expansion to derive the ranges of the indices of this summation:
\begin{equation*}
(D_2D_1)^n(z)=\sum_{i,j,k}r_n(i,j,k)z^i(xy)^j(x+y)^k.
\end{equation*}
It follows from Lemma~\ref{Lemma-Jacobi} that the variable $z$
marks plateaus of $\pi\in\operatorname{JSP}_n$. Hence the degree of $z$ ranges from $1$ to $n$.
Moreover, since the variable $x$ marks ascents and the variable $y$ marks descents, we have $i+2j+k=3n+1$. Set $s_n(i,j)=r_n(i,j,k)$. Then we can write
$(D_4D_3)^n(a)$ as follows:
\begin{equation}\label{D4D3a}
(D_4D_3)^n(a)=\sum_{i=1}^{n}\sum_{j=0}^{\lrf{(3n+1-i)/2}}s_{n}(i,j)a^{i}c^{j}b^{3n+1-i-2j}.
\end{equation}
Then upon taking $a=z,b=x+y$ and $c=xy$, we get~\eqref{mathm02-13}.
In particular,
we have $$s_1(1,1)=1,~s_2(1,2)=3,~s_2(1,3)=4,~s_2(2,1)=1,~s_2(2,2)=8.$$
We now derive a recurrence for $s_n(i,j)$.
For convenience, we set $k=3n+1-i-2j$.
Note that
$$D_3(D_4D_3)^n(a)=\sum_{i,j}s_{n}(i,j)(ia^{i-1}c^{j+1}b^k+ja^ic^jb^{k+1}+2ka^ic^{j+1}b^{k-1}).$$
It follows that
\begin{align*}
&D_4\left(D_3(D_4D_3)^n(a)\right)\\
&=D_4\left(\sum_{i,j}s_{n}(i,j)(ia^{i-1}c^{j+1}b^k+ja^ic^jb^{k+1}+2ka^ic^{j+1}b^{k-1})\right)\\
&=\sum_{i,j}s_{n}(i,j)\left(i(i-1)a^{i-1}c^{j+2}b^k+i(j+1)a^ic^{j+1}b^{k+1}+2ika^ic^{j+2}b^{k-1}\right)+\\
&\sum_{i,j}s_{n}(i,j)\left(ija^ic^{j+1}b^{k+1}+j^2a^{i+1}c^jb^{k+2}+2j(k+1)a^{i+1}c^{j+1}b^k\right)+\\
&\sum_{i,j}s_{n}(i,j)\left(2ika^ic^{j+2}b^{k-1}+2(j+1)ka^{i+1}c^{j+1}b^k+4k(k-1)a^{i+1}c^{j+2}b^{k-2}\right).
\end{align*}
On the other hand,
$(D_4D_3)^{n+1}(a)=\sum_{i,j}s_{n+1}(i,j)a^{i}c^{j}b^{k+3}$.
Comparing the coefficients of $a^ic^jb^{k+3}$ in both sides of $(D_4D_3)^{n+1}(a)=D_4\left(D_3(D_4D_3)^n(a)\right)$, we get the desired recurrence relation.
\end{proof}

We define $$\operatorname{JSP}_{n,k}=\{\pi\in\operatorname{JSP}_n\mid \plat(\pi)=k\}.$$ Let $\vartheta(\pi)$ be the permutation obtained from $\pi\in\operatorname{JSP}_{n}$ by deleting all of the first unbarred $i$ from left to right, where $i\in [n]$. For example, $\vartheta(1331\overline{1}~\overline{2}2442\overline{4}~\overline{3})=31\overline{1}~\overline{2}42\overline{4}~\overline{3}$.
Let $\widehat{\operatorname{JSP}}_{n,n}=\{\vartheta(\pi)\mid\pi\in\operatorname{JSP}_{n,n}\}$. Note that $\#\widehat{\operatorname{JSP}}_{n,n}=(2n)!$.
Then $\vartheta$ is a bijection from $\operatorname{JSP}_{n,n}$ to $\ms_{2n}$.
It is easy to verify that
\begin{equation*}
\sum_{j=1}^{\lrf{(2n+1)/2}}s_{n}(n,j)(xy)^{j}(x+y)^{2n+1-2j}=\sum_{\pi\in\ms_{2n}}x^{\des(\pi)+1}y^{\asc(\pi)+1}.
\end{equation*}

Let $\operatorname{JSPD}_k$ denote the set of Jacobi-Stirling permutations of the multiset $MD_k=M_k\setminus \{k,k\}$,
where a Jacobi-Stirling permutation
of $MD_k$ is a permutation of $MD_k$ such that if $i<j<k$, $\pi_i$ and $\pi_k$ are both unbarred and $\pi_i=\pi_k$, then $\pi_j>\pi_i$.
In particular, $\operatorname{JSPD}_1=\{\overline{1}\}$ and
$\operatorname{JSPD}_2=\{\overline{2}~\overline{1}11,\overline{1}~\overline{2}11,\overline{1}1\overline{2}1,\overline{1}11\overline{2},\overline{2}11\overline{1},
1\overline{2}1\overline{1},11\overline{2}~\overline{1},11\overline{1}~\overline{2}\}$.
For $\pi\in \operatorname{JSPD}_n$, we always set $\pi_0=\pi_{3n-1}=0$.
We define
\begin{equation*}
T_n(x,y,z)=\sum_{\pi\in \operatorname{JSPD}_n}x^{\asc(\pi)}y^{\des(\pi)}z^{\plat(\pi)}.
\end{equation*}
As the proof of Theorem~\ref{mathm02}, it is easy to verify that for $n\geq 1$, we have
\begin{equation*}\label{des-asc-plat-grammar}
D_1(D_2D_1)^{n-1}(x)=D_1(D_2D_1)^{n-1}(y)=D_1(D_2D_1)^{n-1}(z)=T_{n}(x,y,z).
\end{equation*}
Furthermore,
$$T_{n}(x,y,z)=\sum_{i=0}^{n-1}z^{i}\sum_{j=1}^{\lrf{(3n-1-i)/2}}t_n(i,j)(xy)^{j}(x+y)^{3n-1-i-2j}.$$
In particular,
$$\sum_{j=1}^{\lrf{2n-1}/2}t_n(n-1,j)(xy)^{j}(x+y)^{2n-2j}=\sum_{\pi\in\ms_{2n-1}}x^{\des(\pi)+1}y^{\asc(\pi)+1}.$$

We define $$\overline{S}_n(x,y)=\sum_{i=1}^{n}\sum_{j=0}^{\lrf{(3n+1-i)/2}}s_{n}(i,j)x^iy^j,$$
$$\overline{T}_n(x,y)=\sum_{i=0}^{n-1}\sum_{j=1}^{\lrf{(3n-1-i)/2}}t_n(i,j)x^iy^j.$$
It follows from~\eqref{JAcobi-gram02} that
\begin{equation}\label{D4D3a02}
D_3(D_4D_3)^{n-1}(a)=\sum_{i=0}^{n-1}\sum_{j=1}^{\lrf{(3n-1-i)/2}}t_{n}(i,j)a^{i}c^jb^{3n-1-i-2j}.
\end{equation}
Using~\eqref{D4D3a} and~\eqref{D4D3a02}, we get the following result.
\begin{proposition}\label{prop15}
For $n\geq 1$, the numbers $s_n(i,j)$ and $t_n(i,j)$ satisfy the recurrence relation
\begin{align*}
s_n(i,j)&=it_n(i,j-1)+jt_n(i-1,j)+2(3n+2-i-2j)t_n(i-1,j-1),\\
t_{n+1}(i,j)&=(i+1)s_n(i+1,j-1)+js_n(i,j)+2(3n+3-i-2j)s_n(i,j-1),
\end{align*}
with the initial conditions $t_1(0,1)=1$ and $t_1(i,j)=0$ for $(i,j)\neq (0,1)$.
Equivalently, we have
\begin{align*}
\overline{S}_n(x,y)&=2(3n-1)xy\overline{T}_n(x,y)+xy(1-2x)\frac{\partial}{\partial x}\overline{T}_n(x,y)+xy(1-4y)\frac{\partial}{\partial y}\overline{T}_n(x,y),\\
\overline{T}_{n+1}(x,y)&=2(3n+1)y\overline{S}_n(x,y)+y(1-2x)\frac{\partial}{\partial x}\overline{S}_n(x,y)+y(1-4y)\frac{\partial}{\partial y}\overline{S}_n(x,y).
\end{align*}
\end{proposition}

The first few $\overline{S}_n(x,y)$ and $\overline{T}_n(x,y)$ are given as follows:
\begin{align*}
\overline{T}_1(x,y)=y,~
\overline{S}_1(x,y)=xy,~
\overline{T}_2(x,y)=xy+2xy^2+y^2.
\end{align*}
\subsection{Partial $\gamma$-coefficients and a modified Foata-Strehl's group action}
\hspace*{\parindent}

For the grammars~\eqref{JAcobi-gram02}, notice that
the insertion of $\overline{n}$ corresponds to the substitution rules in $G_3$, and
the insertion of the pair $nn$ corresponds to the substitution rules in $G_4$.
Figure~1 provides a diagram of the grammars~\eqref{JAcobi-gram02}. Using this diagram, we discover some statistics on
Jacobi-Stirling permutations, and then we can present combinatorial interpretations of the numbers
$s_n(i,j)$ and $t_n(i,j)$.
\hspace*{\parindent}

\begin{center}
\includegraphics[width=6.0cm,height=4.0cm]{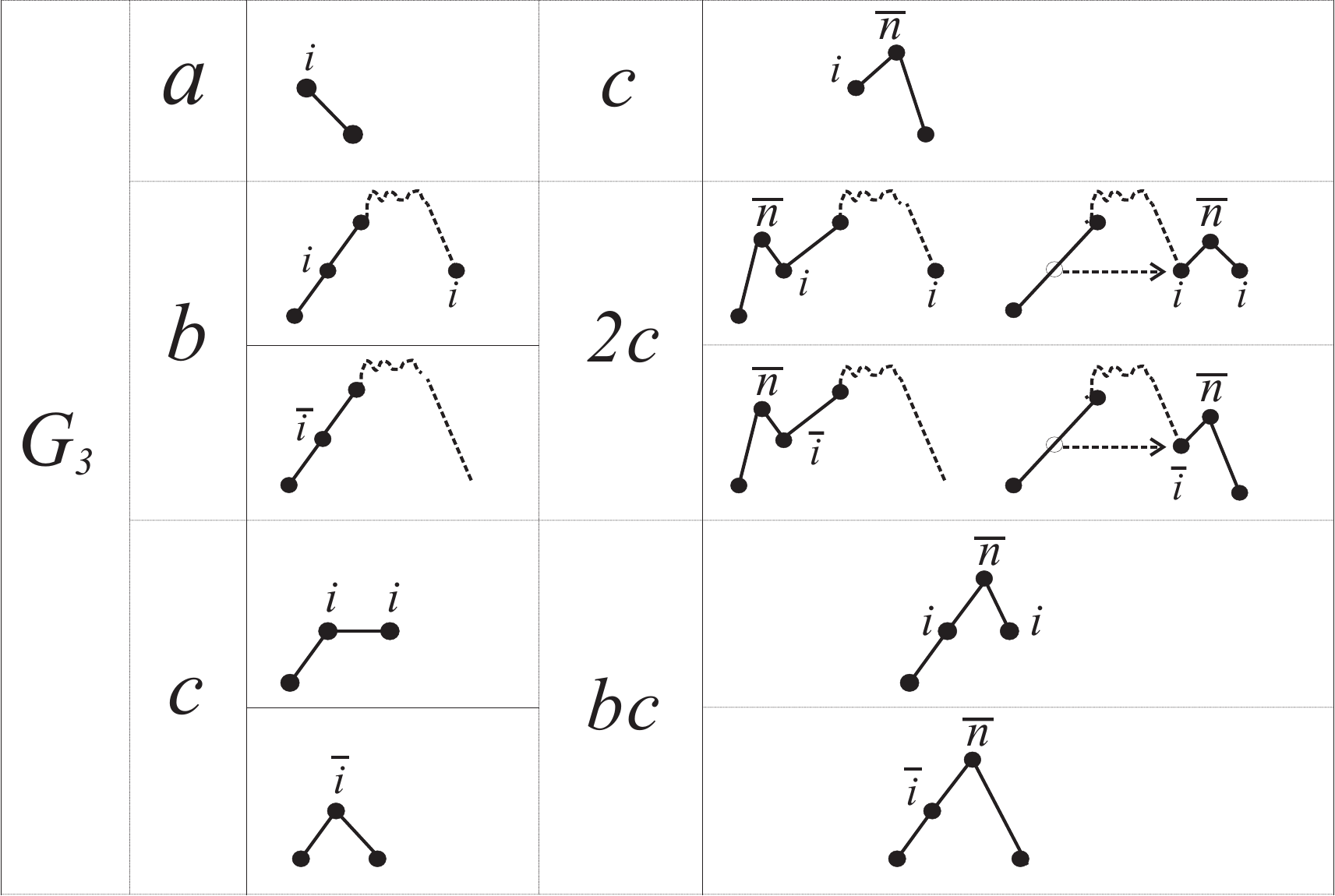}\includegraphics[width=6.0cm,height=4.0cm]{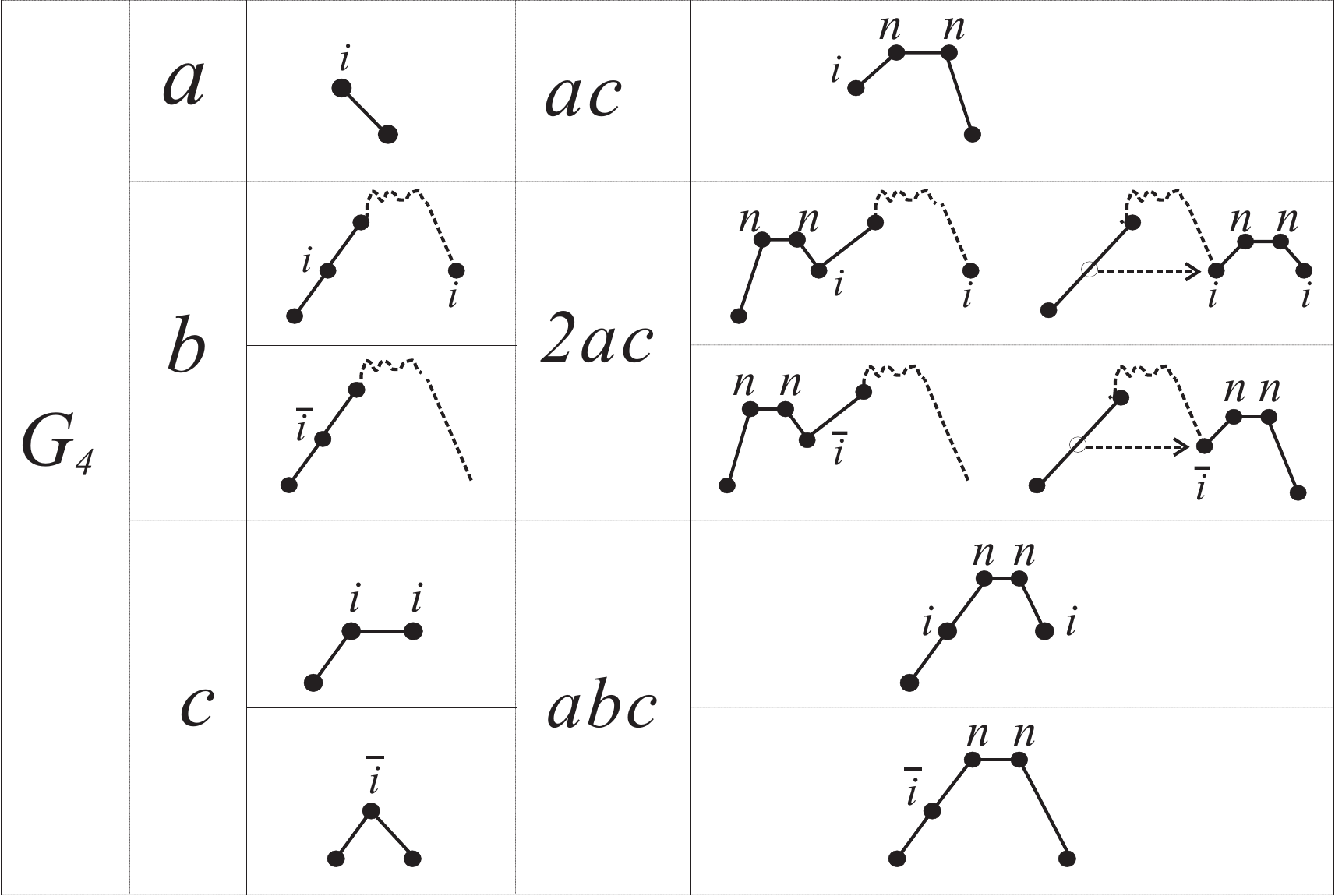}\\
Figure~1.
\end{center}
\hspace*{\parindent}
Let $\pi=\pi_1\pi_2\cdots\pi_{3n}\in \operatorname{JSP}_n$. As usual, we set $\pi_0=\pi_{3n+1}=0$.
An {\it unbarred descent} of $\pi$ is an index $i\in [3n]$ such that $\pi_i>\pi_{i+1}$ and $\pi_i$ is unbarred.
A {\it double ascent} (resp. {\it peak}, {\it left ascent-plateau}) of $\pi$ is an index $i$ such that $\pi_{i-1}<\pi_i<\pi_{i+1}$ (resp.
$\pi_{i-1}<\pi_i>\pi_{i+1}$, $\pi_{i-1}<\pi_i=\pi_{i+1}$), where $i\in[3n]$. It is clear that if $i$ is a peak, then $\pi_i$ is barred.
A {\it barred double descent} of $\pi$ is an index $i\in [3n]$ such that $\pi_{i-1}>\pi_i>\pi_{i+1}$ and $\pi_i$ is barred.
A {\it descent plateau} of $\pi$ is an index $i$ such that $\pi_{i-1}>\pi_i=\pi_{i+1}$.
In the same way, we define the same statistics on $\operatorname{JSPD}_n$.
For $\pi\in\operatorname{JSPD}_n$, we always assume that $\pi_0=\pi_{3n-1}=0$.

We define
\begin{align*}
\ubdes(\pi)&=\#\{i\mid \pi_i>\pi_{i+1},~{\text{$\pi_i$ is unbarred}}\},\\
\dasc(\pi)&=\#\{i\mid \pi_{i-1}<\pi_i<\pi_{i+1}\},\\
\expk(\pi)&=\#\{i\mid \pi_{i-1}<\pi_i>\pi_{i+1}~~{\text{or}}~~\pi_{i-1}<\pi_i=\pi_{i+1}\},\\
\bddes(\pi)&=\#\{i\mid \pi_{i-1}>\pi_i>\pi_{i+1},~{\text{$\pi_i$ is barred}}\},\\
\desp(\pi)&=\#\{i\mid\pi_{i-1}>\pi_i=\pi_{i+1}\}.
\end{align*}

\begin{theorem}\label{snijtnij}
For $n\geq 1$, we have
\begin{align*}
s_{n}(i,j)&=\#\{\pi\in\operatorname{JSP}_n\mid \ubdes(\pi)=i,\expk(\pi)=j,\bddes(\pi)=\desp(\pi)=0\},\\
t_{n}(i,j)&=\#\{\pi\in\operatorname{JSPD}_n\mid \ubdes(\pi)=i,\expk(\pi)=j,\bddes(\pi)=\desp(\pi)=0\}.
\end{align*}
\end{theorem}

The proof of Theorem~\ref{snijtnij} given in the rest of this section follows the same line of argument as that in subsection~\ref{subsection3.3}.

Give a permutation $\pi\in\operatorname{JSPD}_{n}$. For any $k\in\{0,1,\ldots,3n-2\}$,
let $\theta_{n,k}(\pi)$ be the permutation in $\JSP_{n}$ obtained from $\pi$ by inserting the
pair $nn$ between $\pi_k$ and $\pi_{k+1}$, and let $\psi_{\overline{n}}(\pi)$ denote
the permutation in $\operatorname{JSP}_{n-1}$ obtained from $\pi$ by deleting the entry $\overline{n}$.

Give a permutation $\pi\in\operatorname{JSP}_{n}$. For any $k\in\{0,1,\ldots,3n\}$, let $\theta_{\overline{n+1},k}(\pi)$
be the permutation in $\operatorname{JSPD}_{n+1}$ obtained from $\pi$ by inserting the entry $\overline{n+1}$
between $\pi_k$ and $\pi_{k+1}$, and let $\psi_n(\pi)$
denote the permutation in $\operatorname{JSPD}_{n}$ obtained from $\pi$ by deleting the pair $nn$.

We define
\begin{align*}
\operatorname{JSP}_{n;i,j}&=\{\pi\in\operatorname{JSP}_n\mid \ubdes(\pi)=i,\expk(\pi)=j,\bddes(\pi)=\desp(\pi)=0\},\\
\operatorname{JSPD}_{n;i,j}&=\{\pi\in\operatorname{JSPD}_n\mid \ubdes(\pi)=i,\expk(\pi)=j,\bddes(\pi)=\desp(\pi)=0\}.
\end{align*}
For convenience, in the rest of this section we use $|C|$ to denote the cardinality of a set $C$.
For $\pi\in\operatorname{JSPD}_{n;i,j}$, since $\bddes(\pi)=\desp(\pi)=0$, it is easy to verify that
$$|\Dasc(\pi)|=3n-1-|\Ubdes(\pi)|-2|\Expk(\pi)|.$$
Moreover, for $\sigma\in\operatorname{JSP}_{n;i,j}$, we have $|\Dasc(\sigma)|=3n+1-|\Ubdes(\sigma)|-2|\Expk(\sigma)|$.
\begin{lemma}\label{t-s-recurrence}
For $n\geq 1$, we have
\begin{equation}\label{Lemma01RecuJacobi}
|\operatorname{JSPD}_{n+1; i,j}|=(i+1)|\operatorname{JSP}_{n; i+1,j-1}|+j|\operatorname{JSP}_{n; i,j}|+2(3n+3-i-2j)|\operatorname{JSP}_{n; i,j-1}|.
\end{equation}
\end{lemma}
\begin{proof}
We define
\begin{align*}
\Ubdes(\pi)&=\{i\mid \pi_i>\pi_{i+1},~{\text{$\pi_i$ is unbarred}}\},\\
\Expk(\pi)&=\{i\mid \pi_{i-1}<\pi_i>\pi_{i+1}\text{ or }\pi_{i-1}<\pi_i=\pi_{i+1}\},\\
\Dasc(\pi)&=\{i\mid \pi_{i-1}<\pi_i<\pi_{i+1}\}.
\end{align*}
For any $\pi\in\operatorname{JSPD}_{n+1;i,j}$,
let $r=r(\pi)$ be the index such that $\pi_r=\overline{n+1}$.
We now partition the set $\operatorname{JSPD}_{n+1;i,j}$ into the following six subsets:
\begin{align*}
\operatorname{JSPD}_{n+1; i,j}^{1}&=\{\pi\in\operatorname{JSPD}_{n+1;i,j}\mid \pi_{r-1}>\pi_{r+1},~{\text{$\pi_{r-1}$ is unbarred}}\},\\
\operatorname{JSPD}_{n+1; i,j}^{2}&=\{\pi\in\operatorname{JSPD}_{n+1;i,j}\mid \pi_{r-2}<\pi_{r-1}>\pi_{r+1},~{\text{$\pi_{r-1}$ is barred}}\},\\
\operatorname{JSPD}_{n+1; i,j}^{3}&=\{\pi\in\operatorname{JSPD}_{n+1;i,j}\mid  \pi_{r-2}<\pi_{r-1}=\pi_{r+1}\},\\
\operatorname{JSPD}_{n+1; i,j}^{4}&=\{\pi\in\operatorname{JSPD}_{n+1;i,j}\mid \pi_{r-1}<\pi_{r+1}<\pi_{r+2}\},\\
\operatorname{JSPD}_{n+1; i,j}^{5}&=\{\pi\in\operatorname{JSPD}_{n+1;i,j}\mid \pi_{r-2}>\pi_{r-1}=\pi_{r+1}\},\\
\operatorname{JSPD}_{n+1; i,j}^{6}&=\{\pi\in\operatorname{JSPD}_{n+1;i,j}\mid \pi_{r-2}>\pi_{r-1}>\pi_{r+1},~{\text{$\pi_{r-1}$ is barred}}\}.
\end{align*}

{\it Claim 1.} There is a bijection $$\phi_1:\operatorname{JSPD}_{n+1; i,j}^{1}\mapsto \{(\sigma,k)\mid \sigma\in\operatorname{JSP}_{n; i+1,j-1}\text{ and }k\in \Ubdes(\sigma)\}.$$
For any $\pi\in\operatorname{JSPD}_{n+1; i,j}^{1}$, notice that $\psi_{\overline{n+1}}(\pi)\in \operatorname{JSP}_{n; i+1,j-1}$ and $r(\pi)-1\in \Ubdes\left(\psi_{\overline{n+1}}(\pi)\right)$.
Thus, we define the map $\phi_1$ by letting $$\phi_1(\pi)=(\psi_{\overline{n+1}}(\pi),r(\pi)-1).$$
Then the inverse of $\phi_1$ is given by $\phi_1^{-1}(\sigma,k)=\theta_{\overline{n+1},k}(\sigma)$. Therefore, the first term
of the right-hand side of~\eqref{Lemma01RecuJacobi} is explained.

{\it Claim 2.} There is a bijection $$\phi_2:
\operatorname{JSPD}_{n+1; i,j}^{2}\cup\operatorname{JSPD}_{n+1; i,j}^{3}\mapsto \{(\sigma,k)\mid \sigma\in\operatorname{JSP}_{n; i,j}\text{ and }k\in \Expk(\sigma)\}.$$
For any $\pi\in\operatorname{JSPD}_{n+1; i,j}^{2}\cup\operatorname{JSPD}_{n+1; i,j}^{3}$, notice that
$\psi_{\overline{n+1}}(\pi)\in \operatorname{JSP}_{n; i,j}$ and $r(\pi)-1\in \Expk(\psi_{\overline{n+1}}(\pi))$.
Thus, we define the map $\phi_2$ by letting $$\phi_2(\pi)=(\psi_{\overline{n+1}}(\pi),r(\pi)-1).$$
Then the inverse of $\phi_2$ is given by $\phi_2^{-1}(\sigma,k)=\theta_{\overline{n+1},k}(\sigma)$.
Therefore, the second term
of the right-hand side of~\eqref{Lemma01RecuJacobi} is explained.

{\it Claim 3.} There is a bijection $$\phi_3:\operatorname{JSPD}_{n+1; i,j}^{4}\mapsto \{(\sigma,k)\mid \sigma\in\operatorname{JSP}_{n; i,j-1}\text{ and }k\in \Dasc(\sigma)\}.$$
For any $\pi\in\operatorname{JSPD}_{n+1; i,j}^{4}$, notice that $\psi_{\overline{n+1}}(\pi)\in \operatorname{JSP}_{n; i,j-1}$ and $r(\pi)\in \Dasc(\psi_{\overline{n+1}}(\pi))$.
Thus, we define the map $\phi_3$ by letting $\phi_3(\pi)=(\psi_{\overline{n+1}}(\pi),r(\pi))$. Then the inverse of $\phi_3$ is given by $$\phi_3^{-1}(\sigma,k)=\theta_{\overline{n+1},k-1}(\sigma).$$ Then the term $(3n+1-i-2(j-1))|\JSP_{n; i,j-1}|=(3n+3-i-2j)|\JSP_{n; i,j-1}|$ is explained.

{\it Claim 4.} There is a bijection
$$\phi_4:\operatorname{JSPD}_{n+1; i,j}^{5}\cup\operatorname{JSPD}_{n+1; i,j}^{6}\mapsto \{(\sigma,k)\mid \sigma\in\operatorname{JSP}_{n; i,j-1}\text{ and }k\in \Dasc(\sigma)\}.$$
Let $k \in \{0,1,\ldots,3n+1\}$ and let $\pi= \pi_1\pi_2\ldots \pi_{3n+1}\in\operatorname{JSPD}_{n+1}.$
We define a {\it modified Foata-Strehl group action} $\varphi_k$ as follows:
\begin{itemize}\item  If $k$ is a double ascent, then $\varphi_k(\pi)$ is obtained from $\pi$ by deleting $\pi_k$ and then inserting $\pi_k$ immediately before
the integer $\pi_{j}$, where
$j=\min\{s\in \{k+1,k+2,\ldots,3n+2\}\mid\pi_s \leq \pi_k\}$;
\item If $k$ satisfies either $(i)$ it is a descent-plateau or $(ii)$ it is a double descent and $\pi_k$ is barred, then $\varphi_k(\pi)$ is obtained from $\pi$ by deleting $\pi_k$ and then inserting $\pi_k$ right after the integer $\pi_{j}$, where
$j=\max\{s\in \{0,1,2,\ldots,k-1\}\mid\pi_s <\pi_k\}$.
\end{itemize}

For any $\pi\in\operatorname{JSPD}_{n+1; i,j}^{5}$, notice that the index $r(\pi)-1$ is the unique descent-plateau of $\psi_{\overline{n+1}}(\pi)$ and
$\varphi_{r(\pi)-1}\circ \psi_{\overline{n+1}}(\pi)\in \operatorname{JSP}_{n; i,j-1}$.
Read $\varphi_{r(\pi)-1}\circ \psi_{\overline{n+1}}(\pi)$ from left to right and let $p$ be the index of the first occurrence of the integer $\pi_{r(\pi)-1}$. Then $p\in \Dasc(\varphi_{r(\pi)-1}\circ \psi_{\overline{n+1}}(\pi))$.

For any $\pi\in\operatorname{JSPD}_{n+1; i,j}^{6}$, notice that $\pi_{r(\pi)-1}$ has a bar and the index $r(\pi)-1$ is a double-descent in $\psi_{\overline{n+1}}(\pi)$, and $\varphi_{r(\pi)-1}\circ \psi_{\overline{n+1}}(\pi)\in \operatorname{JSP}_{n; i,j-1}$.
Read $\varphi_{r(\pi)-1}\circ \psi_{\overline{n+1}}(\pi)$ from left to right and let $p$ be the index of the occurrence of the integer $\pi_{r(\pi)-1}$. Then $$p\in \Dasc(\varphi_{r(\pi)-1}\circ \psi_{\overline{n+1}}(\pi)).$$
Therefore, we define the map $\phi_4$ by letting $\phi_4(\pi)=(\varphi_{r(\pi)-1}\circ \psi_{\overline{n+1}}(\pi),p)$,
and the inverse of $\phi_4$ is given by $\phi_4^{-1}(\sigma,k)=\theta_{\overline{n+1},r-1}\circ\varphi_k(\sigma)$,
where $r-1$ is the unique descent-plateau or barred double descent of $\varphi_k(\sigma)$.
In conclusion, we also get the term $(3n+3-i-2j)|\operatorname{JSP}_{n; i,j-1}|$.
This completes the proof.
\end{proof}

\begin{lemma}\label{s-t-recurrence}
For $n\geq 1$, we have
\begin{eqnarray*}
|\operatorname{JSP}_{n; i,j}|=i|\operatorname{JSPD}_{n; i,j-1}|+j|\operatorname{JSPD}_{n; i-1,j}|+2(3n+2-i-2j)|\operatorname{JSPD}_{n; i-1,j-1}|.
\end{eqnarray*}
\end{lemma}
For any $\pi\in\operatorname{JSP}_{n;i,j}$,
let $r=r(\pi)$ be the index of the first occurrence of the entry $n$, i.e., $\pi_r=\pi_{r+1}=n$.
We now partition the set $\operatorname{JSP}_{n;i,j}$ into six subsets:
\begin{align*}
\operatorname{JSP}_{n; i,j}^{1}&=\{\pi\in\operatorname{JSP}_{n;i,j}\mid \pi_{r-1}>\pi_{r+2},~{\text{$\pi_{r-1}$ is unbarred}}\},\\
\operatorname{JSP}_{n; i,j}^{2}&=\{\pi\in\operatorname{JSP}_{n;i,j}\mid \pi_{r-2}<\pi_{r-1}>\pi_{r+2},~{\text{$\pi_{r-1}$ is barred}}\},\\
\operatorname{JSP}_{n; i,j}^{3}&=\{\pi\in\operatorname{JSP}_{n;i,j}\mid  \pi_{r-2}<\pi_{r-1}=\pi_{r+2}\},\\
\operatorname{JSP}_{n; i,j}^{4}&=\{\pi\in\operatorname{JSP}_{n;i,j}\mid \pi_{r-1}<\pi_{r+2}<\pi_{r+3}\},\\
\operatorname{JSP}_{n; i,j}^{5}&=\{\pi\in\operatorname{JSP}_{n;i,j}\mid \pi_{r-2}>\pi_{r-1}=\pi_{r+2}\},\\
\operatorname{JSP}_{n; i,j}^{6}&=\{\pi\in\operatorname{JSP}_{n;i,j}\mid \pi_{r-2}>\pi_{r-1}>\pi_{r+2},~{\text{$\pi_{r-1}$ is barred}}\}.
\end{align*}
Then Lemma~\ref{s-t-recurrence} can be proved in the same way as in the proof of Lemma~\ref{t-s-recurrence}.
Here, we only list the following bijections and omit details for simplicity:
\begin{itemize}
  \item [$\Phi_1$:]$\operatorname{JSP}_{n; i,j}^{1}\mapsto \{(\sigma,k)\mid \sigma\in\operatorname{JSPD}_{n; i,j-1}\text{ and }k\in \Ubdes(\sigma)\}$,
  \item [$\Phi_2$:]$\operatorname{JSP}_{n; i,j}^{2}\cup\operatorname{JSP}_{n; i,j}^{3}\mapsto\{(\sigma,k)\mid \sigma\in\operatorname{JSPD}_{n; i-1,j}\text{ and }k\in \Expk(\sigma)\}$,
  \item [$\Phi_3$:]$\operatorname{JSP}_{n; i,j}^{4}\mapsto \{(\sigma,k)\mid\sigma\in\operatorname{JSPD}_{n; i-1,j-1}\text{ and }k\in \Dasc(\sigma)\}$,
  \item [$\Phi_4$:]$\operatorname{JSP}_{n; i,j}^{5}\cup\operatorname{JSP}_{n; i,j}^{6}\mapsto \{(\sigma,k)\mid \sigma\in\operatorname{JSPD}_{n; i-1,j-1}\text{ and }k\in \Dasc(\sigma)\}$.
\end{itemize}

\noindent{\bf A proof
Theorem~\ref{snijtnij}:}
\begin{proof}
Notice that $\operatorname{JSPD}_{1;0,1}=\{\bar{1}\}$ and $\operatorname{JSP}_{1;1,1}=\{\bar{1}11\}$.
Moreover, $\operatorname{JSPD}_{1;i,j}=\emptyset$ for any $(i,j)\neq (0,1)$ and $\operatorname{JSP}_{1;i,j}=\emptyset$ for any $(i,j)\neq (1,1)$. So,
$$t_1(0,1)=1=|\operatorname{JSPD}_{1;0,1}|\text{ and }s_1(1,1)=1=|\operatorname{JSP}_{1;1,1}|.$$
Combining Proposition~\ref{prop15}, Lemma~\ref{t-s-recurrence} and Lemma~\ref{s-t-recurrence}, then by induction we obtain
\begin{align*}
|\operatorname{JSP}_{n; i,j}|&=i|\operatorname{JSPD}_{n; i,j-1}|+j|\operatorname{JSPD}_{n; i-1,j}|+2(3n+2-i-2j)|\operatorname{JSPD}_{n; i-1,j-1}|\\
&=it_n(i,j-1)+jt_n(i-1,j)+2(3n+2-i-2j)t_n(i-1,j-1)\\
&=s_n(i,j),
\end{align*}
\begin{align*}
|\operatorname{JSPD}_{n+1; i,j}|&=(i+1)|\operatorname{JSP}_{n; i+1,j-1}|+j|\operatorname{JSP}_{n; i,j}|+2(3n+3-i-2j)|\operatorname{JSP}_{n; i,j-1}|\\
&=(i+1)s_n(i+1,j-1)+js_n(i,j)+2(3n+3-i-2j)s_n(i,j-1)\\
&=t_{n+1}(i,j).
\end{align*}
This completes the proof.
\end{proof}

Let $[\overline{k}]=\{\overline{1},\overline{2},\ldots,\overline{k}\}$. For any subset $S\subseteq [\overline{k}]$, let
$M_{k,S}=M_k\setminus S$.
Denote by $\operatorname{JSP}_{k,S}$ the set of Jacobi-Stirling permutations of $M_{k,S}$.
Let
$$\operatorname{JSP}_{k,i}=\bigcup_{\substack{S\subseteq [\overline{k}]\\ |S|=i}}\operatorname{JSP}_{k,S}.$$
We define
$$\operatorname{JSP}_{k,i}(x,y,z)=\sum_{\pi\in \operatorname{JSP}_{k,i}}x^{\asc(\pi)}y^{\des(\pi)}z^{\plat(\pi)}.$$
It is clear that $$\operatorname{JSP}_{k,k}(x,y,z)=\sum_{\pi\in \mq_k}x^{\asc(\pi)}y^{\des(\pi)}z^{\plat(\pi)},$$
$$\operatorname{JSP}_{k,0}(x,y,z)=\sum_{\pi\in \operatorname{JSP}_k}x^{\asc(\pi)}y^{\des(\pi)}z^{\plat(\pi)}.$$
Based on empirical evidence, we propose the following conjecture.
\begin{conjecture}
For any $k\geq 1$ and $1\leq i \leq k-1$, the polynomial $\operatorname{JSP}_{k,i}(x,y,z)$
is a partial $\gamma$-positive polynomial.
\end{conjecture}
\section{Derangement polynomials of type $B$}\label{section06}
\subsection{Basic definitions and notation}
\hspace*{\parindent}

Recall that elements $\pi$ of $B_n$ are signed permutations of the set $[\pm n]$ such that
$\pi(-i)=-\pi(i)$ for $i=1,2,\ldots,n$.
As usual, we write signed permutations of $B_n$ as $\pi=\pi(1)\pi(2)\cdots \pi(n)$.
In this section, we denote by $\overline{i}$ the negative element $-i$.
We say that $i\in [n]$ is a {\it weak excedance} of $\pi$ if $\pi(i)=i$ or $\pi(|\pi(i)|)>\pi(i)$ (see~\cite[p.~431]{Bre94}).
Let $\we(\pi)$ be the number of weak excedances of $\pi$.
According to~\cite[Theorem~3.15]{Bre94}, we have $$B_n(x)=\sum_{\pi\in B_n}x^{\we(\pi)}.$$

In the following discussion, we always write $\pi\in B_n$
by using its \emph{standard cycle decomposition}, in which each
cycle is written with its largest entry last and the cycles are written in ascending order of their
last entry. It should be noted that the $n$ letters
appearing in the cycle notation for a signed permutation $\pi\in B_n$ are the
letters $\pi(1),\pi(2),\ldots,\pi(n)$.
A {\it rise} in a cycle of $\pi$ is an index $i$ such that $\pi(i)<\pi(|\pi(i)|)$.
We say that $i$ is a {\it singleton} if $(\overline{i})$ is a cycle of $\pi$, i.e., $\pi(i)=\overline{i}$.
Let $\single(\pi)$ be the number of singletons of $\pi$.
We say that $i$ is a {\it fixed point} of $\pi$ if $(i)$ is a cycle of $\pi$, i.e., $\pi(i)=i$.
\begin{example}
The signed permutation $\pi=\overline{3}51\overline{7}2468\overline{9}$ can be written as
$(\overline{9})(\overline{3},1)(2,5)(4,\overline{7},6)(8)$. Moreover, $\pi$ with only one
singleton $9$ and one fixed point $8$, and $\pi$ has $3$ cycle rises.
\end{example}

As noted by Chow~\cite[p.~819]{Chow09}, the number $\we(\pi)$ equals the sum of the number of cycle rises and the number of fixed points of $\pi$.
We say that $\pi$ is a {\it type $B$ derangement} if $\pi(i)\neq i$ for every $i\in [n]$.
For example, $(\overline{6})(\overline{7},\overline{4})(\overline{3},1)(2,5)\in \mdn_7^B$.
Let $\mdn_n^B$ be the set of all derangements in $B_n$.
Following~\cite{Chow09}, the {\it derangement polynomials of type $B$} are defined by
$$d_0^B(x)=1,~d_n^B(x)=\sum_{\pi\in \mdn_n^B}x^{\we(\pi)}.$$
The first few $d_n^B(x)$ are listed as follows:
$d_1^B(x)=1,~d_2^B(x)=1+4x,~d_3^B(x)=1+20x+8x^2$.

For $\pi\in \mdn_n^B$, we have $$\we(\pi)=\#\{i\in [n]\mid \pi(|\pi(i)|)>\pi(i)\}.$$
We say that $i$ is an {\it anti-excedance} of $\pi$ if $\pi(|\pi(i)|)<\pi(i)$.
Let $\aexc(\pi)$ be the number of anti-excedances of $\pi$.
It is clear that $\we(\pi)+\aexc(\pi)+\single(\pi)=n$.

Let $(c_1,c_2,\ldots,c_i)$ be a cycle of $\pi$. Set $c_{i+1}=c_1$ and $c_0=c_i$. Then we say that $c_j$ is called
\begin{itemize}
  \item a {\it cycle ascent} in the cycle if $c_j<c_{j+1}$, where $1\leq j<i$;
  \item a {\it cycle descent} in the cycle if $c_j>c_{j+1}$, where $1\leq j\leq i$;
  \item a {\it cycle double ascent} in the cycle if $c_{j-1}<c_j<c_{j+1}$, where $1<j<i$;
  \item a {\it cycle double descent} in the cycle if $c_{j-1}>c_j>c_{j+1}$, where $1<j<i$;
 \item a {\it cycle peak} in the cycle if $c_{j-1}<c_j>c_{j+1}$, where $1<j \leq i$;
 \item a {\it cycle valley} in the cycle if $c_{j-1}>c_j<c_{j+1}$, where $1\leq j<i$.
\end{itemize}
Clearly, the number of cycle ascents of $\pi$ is just the number of cycle rises.
Denote by $\cda(\pi)$ (resp.~$\cdd(\pi),\cpk(\pi),\cval(\pi)$) the number of cycle double ascents (resp.~cycle double descents, cycle peaks, cycle valleys) of $\pi$.
For $\pi\in \mdn_n^B$,
it is easy to verify that
\begin{equation}\label{st-der}
\we(\pi)=\cpk(\pi)+\cda(\pi),~\aexc(\pi)=\cval(\pi)+\cdd(\pi),~\cpk(\pi)=\cval(\pi).
\end{equation}
\subsection{Main results}
\hspace*{\parindent}

Let $$E_n(x,y,z)=\sum_{\pi\in \mdn_n^B}x^{\we(\pi)}y^{\aexc(\pi)}z^{\single(\pi)}.$$
Very recently, we obtained the following lemma.
\begin{lemma}[{\cite{Ma1801}}]\label{lemma01}
If $A=\{x,y,z,e\}$ and
\begin{equation}\label{grammar-derangement}
G=\{x\rightarrow xy^2,y\rightarrow x^2y, z\rightarrow x^2y^2z^{-3}, e\rightarrow ez^4\},
\end{equation}
then
\begin{equation}\label{Dne-derangement}
D^n(e)=eE_n(x^2,y^2,z^4).
\end{equation}
\end{lemma}

Now we present the fourth main result of this paper.
\begin{theorem}
The polynomial $E_n(x,y,z)$ is a partial $\gamma$-positive polynomial. More precisely,
for $n\geq 0$, we have
\begin{equation}\label{dnbx}
E_n(x,y,z)=\sum_{i=0}^nz^{i}\sum_{j =0}^{\lrf{(n-i)/2}}g_n(i,j)(xy)^{j}(x+y)^{n-i-2j},
\end{equation}
where the numbers $g_n(i,j)$ satisfy the recurrence relation
\begin{equation}\label{gnij-recu}
g_{n+1}(i,j)=g_n(i-1,j)+4(1+i)g_n(i+1,j-1)+2j g_n(i,j)+4(n+2-i-2j)g_n(i,j-1),
\end{equation}
with the initial conditions $g_1(1,0)=1$ and $g_1(1,j)=0$ for $j\neq 0$.
\end{theorem}
\begin{proof}
Consider the grammar~\eqref{grammar-derangement}. If we set $s=e,t=z^4,u=xy$ and $v=x^2+y^2$,
then
$$D(s)=st,D(t)=4u^2,~D(u)=uv,~D(v)=4u^2.$$
Thus, if $A=\{s,t,v,u\}$ and
\begin{equation}\label{grammar-der02}
G=\{s\rightarrow st,t\rightarrow 4u^2,u\rightarrow uv,v\rightarrow 4u^2\},
\end{equation}
then we have
$D(s)=st,~
D^2(s)=s(t^2+4u^2),~
D^3(s)=s(t^3+12tu^2+8u^2v)$.
In general, there exist nonnegative integers $f_n(i,j,k)$ such that
\begin{equation}\label{Dns0111}
D^n(s)=s\sum_{i,j,k}f_n(i,j,k)t^{i}u^{j}v^{k}.
\end{equation}
By induction, it is easy to verify that the indices $i,j$ and $k$ of the summation~\eqref{Dns0111} are all nonnegative integers satisfying $i+j+k=n$. Moreover,
from~\eqref{grammar-der02}, we see that
the variable $u$ has only even powers. Set $g_n(\alpha,\beta)=f_n(\alpha,2\beta,k)$.
Then~\eqref{Dns0111} can be written as follows:
\begin{equation}\label{Dns01}
D^n(s)=s\sum_{\alpha=0}^n\sum_{\beta =0}^{\lrf{(n-\alpha)/2}}g_n(\alpha,\beta)t^{\alpha}u^{2\beta}v^{n-\alpha-2\beta}.
\end{equation}
Comparing~\eqref{Dne-derangement} with~\eqref{Dns01}, we get~\eqref{dnbx}.
Notice that
\begin{align*}
D^{n+1}(s)&=D(D^n(s))\\
&=D\left(s\sum_{\alpha=0}^n\sum_{\beta =0}^{\lrf{(n-\alpha)/2}}g_n(\alpha,\beta)t^{\alpha}u^{2\beta}v^{n-\alpha-2\beta}\right)\\
&=s\sum_{\alpha,\beta}g_n(\alpha,\beta)(t^{\alpha+1}u^{2\beta}v^{n-\alpha-2\beta}+4\alpha t^{\alpha-1}u^{2\beta+2}v^{n-\alpha-2\beta}\\
&+2\beta t^{\alpha}u^{2\beta}v^{n+1-\alpha-2\beta}+4(n-\alpha-2\beta)t^{\alpha}u^{2\beta+2}v^{n-1-\alpha-2\beta}).
\end{align*}
By comparing coefficients on the both sides of $D^{n+1}(s)=D(D^n(s))$, we get~\eqref{gnij-recu}.
\end{proof}

Let $\overline{E}_n(x,y)=\sum_{i,j} g_n(i,j)x^iy^j$.
Multiplying both sides of the recurrence relation~\eqref{gnij-recu} by $x^iy^j$ and summing over all $i,j$, we get that
the polynomials $\overline{E}_n(x,y)$ satisfy the recurrence relation
$$\overline{E}_{n+1}(x,y)=(x+4ny)\overline{E}_n(x,y)+4y(1-x)\frac{\partial}{ \partial x}\overline{E}_n(x,y)+2y(1-4y)\frac{\partial}{\partial y}\overline{E}_n(x,y),$$
with the initial condition $g_0(x,y)=1$. The first few $\overline{E}_n(x,y)$ are given as follows:
$$\overline{E}_1(x,y)=x,~
\overline{E}_2(x,y)=x^2+4y,
~\overline{E}_3(x,y)=x^3+12xy+8y.$$

Let $a_n(x)=\sum_{k\geq 1}a(n,k)x^k$, where the numbers $a(n,k)$ are defined by~\eqref{Anx-gamma}.
\begin{corollary}
For $\geq 0$, we have
\begin{equation*}
\overline{E}_{n+1}(x,y)=x\overline{E}_n(x,y)+\sum_{k=0}^{n-1}\binom{n}{k}2^{n+1-k}\overline{E}_k(x,y)a_{n-k}(y).
\end{equation*}
\end{corollary}
\begin{proof}
Let $G$ be the grammar~\eqref{grammar-der02}.
Note that
$D(t)=4u^2,D^2(t)=8u^2v$.
In general, suppose that $D^n(t)=2^{n+1}\sum_{k\geq 1}\widehat{a}(n,k)u^{2k}v^{n+1-2k}$.
Then we have
\begin{align*}
D^{n+1}(t)&=D\left(2^{n+1}\sum_{k\geq 1}\widehat{a}(n,k)u^{2k}v^{n+1-2k}\right)\\
&=2^{n+2}\sum_{k\geq 1}\widehat{a}(n,k)\left(ku^{2k}v^{n+2-2k}+2(n+1-2k)u^{2k+2}v^{n-2k}\right).
\end{align*}
Therefore,
$\widehat{a}(n+1,k)=k\widehat{a}(n,k)+2(n+3-2k)\widehat{a}(n,k-1)$.
It is clear that $\widehat{a}(1,1)=1$ and $\widehat{a}(1,k)=0$ for $k\neq 1$.
Since $a(n,k)$ and $\widehat{a}(n,k)$ satisfy the same recurrence relation and initial conditions, they agree.
Using the {\it Leibniz's formula},
we get
$$D^{n+1}(s)=D^n(st)=tD^n(s)+\sum_{k=0}^{n-1}\binom{n}{k}D^k(s)D^{n-k}(t),$$
which yields the desired recurrence relation.
\end{proof}

Let $$g_n=\sum_{i=0}^n\sum_{j =0}^{\lrf{(n-i)/2}}g_n(i,j).$$
The first few $g_n$ are $g_0=1,g_1=1,g_2=5,g_3=21,g_4=153,g_5=1209$.
It should be noted that the numbers $g_n$ appear as A182825 in~\cite{Sloane}.

\begin{theorem}
For $n\geq 1$, we have
\begin{equation*}\label{gnij-combinatorial}
g_n(i,j)=\#\{\pi\in \mdn_n^B\mid \single(\pi)=i,\cpk(\pi)=j,\cda(\pi)=0\}.
\end{equation*}
\end{theorem}
\begin{proof}
We define an action $\varphi_{x}$ on $\mdn_n^B$ as follows.
Let $c=(c_1,c_2,\ldots,c_i)$ be a cycle of $\pi\in \mdn_n^B$. Since $c_i=\max\{c_1,c_2,\ldots,c_i\}$, we set
$c_0=c_i,c_{i+1}=c_1$ and $\widetilde{c}=(c_0,c_1,c_2\ldots,c_i,c_{i+1})$.
\begin{itemize}\item  If $c_k$ is a cycle double ascent in $c$,
then $\varphi_{c_k}(\widetilde{c})$ is obtained by deleting $c_k$ and then inserting $c_k$ between $c_j$ and $c_{j+1}$, where $j$ is the largest index satisfying $0\leq j<k$ and $c_j>c_k>c_{j+1}$;
\item If $c_k$ is a cycle double descent in $c$, then $\varphi_{c_k}(\widetilde{c})$ is obtained by deleting $c_k$ and then inserting $c_k$ between $c_j$ and $c_{j+1}$, where $j$ is the smallest index satisfying $k<j<i$ and $c_j<c_k<c_{j+1}$;
\item If $c_k$ is neither a cycle double ascent nor a cycle double descent in $c$, then $c_k$ is a cycle peak or a cycle valley. In this case, we let $\varphi_{c_k}(\widetilde{c})=\widetilde{c}$.
\end{itemize}

We now define a modified Foata-Strehl group action $\varphi'_x$ on $\mdn_n^B$ by
$$\varphi'_x(\pi)=\left\{\begin{array}{lll}
\varphi_x(\pi),&\text{ if $x$ is a cycle double ascent or a cycle double descent;}\\
\pi,&\text{if $x$ is a cycle peak or a cycle valley.}\\
\end{array}\right.$$
It is clear that the ${\varphi'_x}$'s are involutions and that they commute.
For any subset $S\subseteq[n]$, we define the function $\varphi'_S : \mdn_n^B \mapsto \mdn_n^B$ by
$\varphi'_S(\pi)=\prod\limits_{x\in S}\varphi'_x(\pi)$.
Hence the group $\mathbb{Z}^{n}_2$ acts on $\mdn_n^B$ via the function $\varphi'_S$, where $S \subseteq[n]$.
Let $\Orb(\pi)=\{g(\pi):g\in \mathbb{Z}^{n}_2\}$ be
the orbit of $\pi$ under the modified Foata-Strehl group action. Then the modified Foata-Strehl group action
divides the set $\mdn_n^B$ into disjoint orbits such that there is a unique permutation in each orbit which has no cycle double ascent.
Let $\Dasc(\pi)$ and $\Ddes(\pi)$ denote the sets of cycle double ascents and cycle double descents of $\pi$, respectively.
Let $S=S(\pi)=\Dasc(\pi)\cup \Ddes(\pi)$.
Note that $\Dasc(\varphi'_S(\pi))=\Ddes(\pi),~\Ddes(\varphi'_S(\pi))=\Dasc(\pi)$.
Let $\widehat{\pi}$
be the unique element in $\Orb(\pi)$ with no cycle double ascent.
Therefore, using~\eqref{st-der}, it is clear that
\begin{align*}
\sum_{\sigma\in \Orb(\pi)}x^{\we(\sigma)}y^{\aexc(\sigma)}z^{\single(\sigma)}&=\sum_{\sigma\in \Orb(\pi)}z^{\single(\sigma)}(xy)^{\cpk(\sigma)}x^{\cda(\sigma)}y^{\cdd(\sigma)}\\
&=z^{\single(\widehat{\pi})}(xy)^{\cpk(\widehat{\pi})}(x+y)^{n-\single(\widehat{\pi})-2\cpk(\widehat{\pi})}.
\end{align*}
Note that $\cpk(\sigma)=\cpk(\pi)$ for any $\sigma\in \Orb(\pi)$. Thus,
\begin{align*}
\sum_{\sigma\in \Orb(\pi)}x^{\we(\sigma)}y^{\aexc(\sigma)}z^{\single(\sigma)}=z^{\single({\pi})}(xy)^{\cpk({\pi})}(x+y)^{n-\single({\pi})-2\cpk({\pi})},
\end{align*}
and the theorem follows.
\end{proof}

\section{Concluding remarks}
\hspace*{\parindent}
In this paper we apply the change of grammars method to study $\gamma$-positivity and partial $\gamma$-positivity of descent-type polynomials.
Based on the results of this paper, we see that if the grammar of a combinatorial structure is partial symmetric,
then we may use~\eqref{change-grammars} as the type of change of grammars.
In the same way, one may consider multivariate extensions of orthogonal polynomials and rook polynomials.
Moreover, it would be interesting to introduce some partial $\gamma$-positive enumerative polynomials for the
complex reflection groups $G(r,p,n)$, where $r,p,n$ are positive integers such that $p$ divides $r$.
In particular, $G(1,1,n)=\msn$ and $G(2,1,n)=B_n$. Furthermore, $\gamma$-positivity has been extensively
studied in topological combinatorics (see~\cite{Athanasiadis17,Gal05} for instance), it would be interesting to
explore algebraic and topological significance of partial $\gamma$ coefficients and their $q$-analogues.

We end our paper by proposing the following conjecture.
\begin{conjecture}
If $f(x)$ is a real-rooted enumerative polynomial with only nonnegative coefficients,
then there exists a partial $\gamma$-positive polynomial
$f(x,y,z)$ such that $f(x,1,1)=f(x)$.
\end{conjecture}
\section*{Acknowledgements}
The authors appreciate the careful review, corrections and helpful suggestions to this paper made by the referees.

\end{document}